\documentclass[11pt]{article}
\usepackage{amsmath,amsthm,amsfonts,amssymb,amscd,amsxtra}
\usepackage{algorithmic} 
\usepackage{algorithm}
\usepackage{url}
\usepackage{graphicx}
\usepackage{subcaption}
\usepackage{xcolor}
\usepackage{cite}
\usepackage{booktabs}
\usepackage[margin=2.5cm,nohead]{geometry}
\theoremstyle{plain}
\newtheorem{theorem}{Theorem}
\newtheorem{lemma}{Lemma}
\newtheorem{corollary}{Corollary}
\newtheorem{proposition}{Proposition}

\theoremstyle{definition}
\newtheorem{definition}{Definition}

\newtheorem{remark}{Remark}
\newtheorem{example}{Example}
\newtheorem{assumption}{Assumption}

\DeclareMathOperator{\diag}{diag}
\DeclareMathOperator{\sgn}{sgn}

\newcommand{\authorinfo}[2]{%
  \thanks{#1. Email: \protect\url{#2}.}%
}

\graphicspath{{figs/}{figs_revised/}}

\begin{document}

\title{A Safeguarded Projected-Gradient Framework for Complementarity Constrained Least Squares Problems}

\author{%
Lianghai Xiao\authorinfo{College of Information Science and Technology,
Jinan University}{xiaolh@jnu.edu.cn}
\and Wei Zhang\authorinfo{Corresponding Author, School of Mathematics, South China University of Technology}{zw2020@scut.edu.cn}
\and Jiayi Zhong\authorinfo{College of Information Science and Technology,
Jinan University}{zjyeexo@stu2024.jnu.edu.cn}}

\date{}
\maketitle
\begin{abstract} 
Both generalized absolute value equations (GAVEs) and linear complementarity
problems (LCPs) can  be formulated as least-squares problems over the complementarity
set. 
Because this feasible set is nonconvex, 
projected stationarity does not in general imply zero residual. We develop
a safeguarded projected-gradient framework with a linear-system refinement
on a selected polyhedral face. 
We establish matrix conditions under which every
projected stationary point is a global solution and derive corresponding
global convergence results.
Specifically, when refinement on a correct active face returns a solution, we give
explicit iteration bounds for active-face identification and prove finite
termination. 
Numerical experiments on
GAVE and LCP benchmarks demonstrate the high accuracy of the proposed
framework across the tested settings.
\end{abstract}
\noindent
\section{Introduction} 

Complementarity systems arise in equilibrium models, constrained
optimization, and piecewise-linear equations
\cite{cottle1992linear,facchinei2003finite,hladik2026overview}. Their defining disjunction---at
most one member of each complementary pair can be positive---provides useful
structure but also makes the feasible set nonconvex. To alleviate this difficulty, most existing methods try to reformulate the complementary structure. Here, we directly consider the
complementarity-constrained least-squares problem
\begin{equation}\label{eq:general_model}
    \min_{(u,v)\in\mathcal C}
    \Phi(u,v)
    :=
    \frac12\|Pu+Qv-c\|^2,
\end{equation}
where $P,Q\in\mathbb R^{m\times n}$, $c\in\mathbb R^m$, and
\[
    \mathcal C
    :=
    \left\{
    (u,v)\in\mathbb R^n\times\mathbb R^n:
    u\geq0,\;
    v\geq0,\;
    u\odot v=0
    \right\},
\]
with $\odot$ denoting the Hadamard product. 
The zero-level set of \eqref{eq:general_model} consists precisely of the
complementary solutions of the linear system
\[
    Pu+Qv=c,
    \qquad
    u\geq0,\quad v\geq0,\quad u\odot v=0.
\]
The following two instances motivate our development.

\begin{example}[GAVE]\label{ex:GAVE}
Consider the generalized absolute value equation
\begin{equation}\label{eq:gave}
    Ax-B|x|=b,
    \qquad x\in\mathbb R^n,
\end{equation}
where $A,B\in\mathbb R^{m\times n}$ and $b\in\mathbb R^m$. It is represented by \eqref{eq:general_model} with
\[
    u=x^+,\qquad
    v=x^-,
    \qquad
    P=A-B,\qquad
    Q=-A-B,\qquad
    c=b.
\]
A zero-residual feasible point therefore yields a solution
$x=x^+-x^-$ of \eqref{eq:gave}.
\end{example}

\begin{example}[LCP]\label{ex:LCP}
Consider the linear complementarity problem
\begin{equation}
\label{eq:lcp}
    z\geq0,\qquad
    w=Mz-b\geq0,\qquad
    z\odot w=0,
\end{equation}
where $M\in\mathbb R^{n\times n}$ and $b\in\mathbb R^n$. It can be written
as an instance of \eqref{eq:general_model} by setting
\[
    u=z,\qquad
    v=w,\qquad
    P=M,\qquad
    Q=-I,\qquad
    c=b.
\] 
\end{example}

\subsection{Related work}

GAVEs and LCPs are closely related classes of piecewise-linear systems.
The connection between absolute value equations and complementarity
problems is developed in \cite{mangasarian2006absolute}, while the theory
and classical algorithms for LCPs are treated systematically in
\cite{cottle1992linear}; see also the recent AVE survey
\cite{hladik2026overview}. Despite their close relationship, the numerical
literatures on GAVEs and LCPs often emphasize different algorithmic tools.

For AVEs and GAVEs, a major line of research applies generalized,
semismooth, or smoothing Newton methods to the piecewise-linear equation
\cite{mangasarian2009generalized,bello2016global,chen2025nonmonotone}.
Picard iterations and matrix-splitting schemes instead use fixed-point or
relaxed linear solvers \cite{rohn2014iterative,zhou2021newton}. Optimization
reformulations have also led to successive linearization, complementarity
smoothing, and gradient-type methods
\cite{mangasarian2007absolute,abdallah2018solving,yang2023modified}. These
approaches are effective under suitable regularity, contractivity, or
unique-solvability conditions. Many of them are formulated for square systems,
whereas alternating-projection and randomized-sketching methods explicitly
accommodate rectangular GAVEs
\cite{alcantara2023method,xie2025randomized}.
For the square AVE $Ax-|x|=b$, finite convergence of generalized Newton
schemes has also been established under conditions tied to the signed
linearizations, including interval regularity of $[A-I,A+I]$
\cite{zhang2009global} and
conditions ensuring its convergence \cite{guo2025comments}.
Such results are tailored to the equation reformulation and do not directly
extend to a general GAVE or to a projected method on the complementarity set.

LCP algorithms include pivoting methods, projected stationary iterations,
and equation-based methods; comprehensive accounts are given in
\cite{cottle1992linear,facchinei2003finite}. Projected Gauss--Seidel and
projected SOR exploit matrix structure and are classical choices for
structured LCPs \cite{cryer1971solution,cottle1992linear}. More generally,
first-order methods are typically analyzed through asymptotic convergence. Still,
there are specialized finite-termination results: for example, a conjugate
gradient method finitely reaches the exact solution for LCPs with a symmetric M-matrix
\cite{li2008conjugate}, while projection-type methods for variational
inequalities can converge finitely under additional geometric conditions such
as weak sharpness \cite{alhomidan2017finite}. 
Neither directly provides an active-face refinement
for the nonconvex least-squares model \eqref{eq:general_model}.

Another influential approach replaces complementarity by a nonsmooth or
smoothed equation, using, for example, the minimum or Fischer--Burmeister
function, and then applies a semismooth or smoothing Newton method
\cite{fischer1992special,kanzow1996global,qi2000newlook}. In addition to
asymptotic superlinear or quadratic convergence, finite termination has been
established for particular Newton-type LCP algorithms. Fischer and Kanzow
proved a local one-step termination property for a nonsmooth Newton-type
iteration and used it to establish finite termination of an existing LCP algorithm
\cite{fischer1996finite}; finite termination was also proved for a damped
Newton algorithm \cite{sun1998finite}. For LCPs with a sufficient matrix, Sun
and Huang developed a smoothing Newton algorithm that finitely reaches a
maximally complementary solution \cite{sun2006smoothing}. These results rely
on the specific equation reformulation and its regularity or matrix
assumptions.

Projection methods provide a closer link between the two literatures.
Alcantara et al. reformulated a general GAVE as a nonconvex
feasibility problem and applied alternating projections
\cite{alcantara2023method}. Alcantara and Lee subsequently developed a
union-convex analysis for several projection methods, proving convergence
to the solution set for P-matrix LCPs and introducing acceleration schemes
based on the active union component \cite{alcantara2025global}. These results
highlight the usefulness of the union-convex structure. Their methods are
built on projections associated with particular feasibility reformulations,
and their analysis is formulated in terms of the corresponding sets and
fixed-point mappings. To the best of our knowledge, no existing framework provides both an
explicit iteration bound for active-face identification and a finite-
termination guarantee.

Motivated by these developments, we reformulate both GAVEs and LCPs as
linear least-squares problems over the complementarity set and develop a
common projected-gradient framework for solving them. The projected-gradient
iteration provides globalization, while a refinement step solves the linear
system associated with a selected polyhedral face. Under explicit conditions,
the refinement produces a zero-residual solution once an appropriate face is
selected, thereby yielding finite termination.

\subsection{Contributions}

The main contributions are as follows.
\begin{itemize}
    \item We develop a safeguarded projected-gradient framework for GAVEs
    and LCPs that combines a common first-order iteration with a safeguarded
    linear-system refinement on a selected polyhedral face.

    \item We establish verifiable conditions under which every projected
    stationary point has zero residual and is therefore a global solution.
    We further establish global convergence to the zero-residual solution set,
    entire-sequence convergence in unique-solution settings, and linear
    convergence to the solution set under error-bound conditions.

    \item We quantify active-face identification and finite termination.
    In particular, we derive explicit iteration bounds after which the faces
    containing the projected-gradient points intersect the
    solution set. When the selected-face linear system returns a solution on
    such a face, we further derive an explicit iteration bound within which the safeguarded
    refinement reaches a zero-residual solution.
\end{itemize}
\subsection{Organization}

The remainder of this paper is organized as follows.
Section~\ref{sec:preliminaries} introduces the complementarity projection,
stationarity measure, and active-face notation.
Section~\ref{sec:algorithmic_framework} presents the safeguarded
projected-gradient framework. Section~\ref{sec:convergence} gives conditions
under which projected stationary points solve \eqref{eq:general_model} and
develops the global convergence analysis. Moreover, using Hoffman error 
bounds, it further derives linear
convergence rates and quantitative bounds for active-face identification and
finite termination.
Section~\ref{sec:numerics} describes the implementation and numerical
experiments. Section~\ref{sec:conclusion} concludes the paper.
\section{Preliminaries}\label{sec:preliminaries}
Throughout the paper, $\|\cdot\|$ denotes the Euclidean norm for vectors and the
spectral norm for matrices. For $W\in\mathbb R^{m\times n}$,
$\sigma_{\min}(W)$ denotes the smallest of its $\min\{m,n\}$ singular values,
and $\operatorname{Range}(W)$ denotes its column space. For a nonempty set
$\Omega$, define
\(
    \operatorname{dist}(y,\Omega):=\inf_{z\in\Omega}\|y-z\|.
\)
For two nonempty sets $\Omega_1$ and $\Omega_2$, define
\(
    \operatorname{dist}(\Omega_1,\Omega_2)
    :=\inf\{\|z_1-z_2\|:z_1\in\Omega_1,\ z_2\in\Omega_2\}.
\)
If $\Omega$ is also closed, let
\(
    \Pi_\Omega(y):=\operatorname*{argmin}_{z\in\Omega}\|y-z\|
\)
denote its Euclidean projection; this mapping may be set-valued when
$\Omega$ is nonconvex.  The symbol $I$ denotes an identity matrix whose
dimension is determined by context. We use $\mathbb N:=\{0,1,2,\ldots\}$.
For an index set $\mathcal I$, $W_{:,\mathcal I}$ denotes the submatrix formed
by the columns indexed by $\mathcal I$, while $W_{\mathcal I\mathcal I}$
denotes the corresponding principal submatrix when $W$ is square.
For $t\in\mathbb R$, let $t^+:=\max\{t,0\}$,
$t^-:=\max\{-t,0\}$, and let $\operatorname{sgn}(t)$ equal $1$, $0$, or
$-1$ according as $t>0$, $t=0$, or $t<0$, respectively. Vector operations
and inequalities are understood componentwise; in particular,
$x=x^+-x^-$ and $|x|=x^++x^-$.
For $y=(u,v)\in\mathcal C$, define
\[
\begin{aligned}
    \mathcal I_{+0}(y)&:=\{i:u_i>0,\ v_i=0\},\\
    \mathcal I_{0+}(y)&:=\{i:u_i=0,\ v_i>0\},\\
    \mathcal I_{00}(y)&:=\{i:u_i=v_i=0\}.
\end{aligned}
\]
Indices in $\mathcal I_{00}(y)$ are called \emph{biactive}.
\begin{definition}
The point $y$ is \emph{strictly complementary} if the biactive set
$\mathcal I_{00}(y)=\varnothing$.
\end{definition}
For a point $y=(u,v)\in\mathcal C$, define
\[
    \vartheta_i(y)
    :=\sgn(u_i-v_i)
    =
    \begin{cases}
        1,&u_i>0,\\
       -1,&v_i>0,\\
        0,&u_i=v_i=0,
    \end{cases}
    \qquad i=1,\ldots,n.
\]
Thus, $\vartheta_i(y)=0$ if and only if
$i\in\mathcal I_{00}(y)$.
For a strictly complementary point $y$, define the active matrix
\(
    \mathcal H(y):=
    [h_1(y) ~\cdots ~h_n(y)],
\)
where
\(
    h_i(y):=
    \begin{cases}
        P_{:i},&i\in\mathcal I_{+0}(y),\\
        Q_{:i},&i\in\mathcal I_{0+}(y).
    \end{cases}
\).
To select one face at a biactive point, assign each biactive coordinate to
the first axis and define the face-selection vector
\[
    \chi_i(y)
    :=
    \begin{cases}
        1,&u_i>0\ \text{or}\ u_i=v_i=0,\\
       -1,&v_i>0.
    \end{cases}
\]
The deterministic active face selected at $y$ is
\[
    F(y)
    :=
    \left\{(\widetilde u,\widetilde v)\in\mathcal C:
    \widetilde v_i=0\ \text{if }\chi_i(y)=1,\quad
    \widetilde u_i=0\ \text{if }\chi_i(y)=-1
    \right\}.
\]
More explicitly, for each $s\in\{-1,1\}^n$, define
\[
    F_s
    :=
    \left\{(\widetilde u,\widetilde v)\in\mathcal C:
    \widetilde v_i=0\ \text{if }s_i=1,\quad
    \widetilde u_i=0\ \text{if }s_i=-1
    \right\},
\]
and let
\begin{equation}
\label{eq:finite_face_collection}
    \mathfrak F:=\{F_s:s\in\{-1,1\}^n\}.
\end{equation}
Thus $\mathfrak F$ is a finite collection of $2^n$ polyhedral cones and
$F(y)=F_{\chi(y)}\in\mathfrak F$.  If $y$ is strictly complementary, then
$F(y)$ is the unique face selected by its nonzero components; the
deterministic rule only resolves the face selection at biactive indices.
For $s\in\{-1,1\}^n$, define
\(
    D_s:=\operatorname{diag}(s)
\) and
\(
    \Theta_s:=\tfrac12(I+D_s).
\)
Then
\begin{equation}
\label{eq:face_signed_parameterization}
    F_s
    =
    \left\{
    \bigl(\Theta_s\xi,-(I-\Theta_s)\xi\bigr):
    D_s\xi\geq0
    \right\}.
\end{equation}
Write $r(y):=Pu+Qv-c$ and
$\mathcal S:=\{y\in\mathcal C:r(y)=0\}$. When
$\mathcal S\neq\varnothing$, we have
\(
    \mathcal S=\underset{y\in\mathcal C}{\operatorname{argmin}}\,\Phi(y),
\) and 
\(
    \min_{y\in\mathcal C}\Phi(y)=0.
\)
For $a\leq b$, let
$\mathcal D_{[a,b]}:=\{\diag(d):d\in[a,b]^n\}$. For
\(\Theta\in\mathcal D_{[0,1]}\), define the matrix
\begin{equation}
\label{eq:matrix_family}
    C_\Theta:=P\Theta-Q(I-\Theta),
\end{equation}
and the uniform matrix modulus
\[
    \mu_C:=\inf_{\Theta\in\mathcal D_{[0,1]}}\sigma_{\min}(C_\Theta).
\]
When $m\leq n$, the condition $\mu_C>0$ means that every $C_\Theta$ has
full row rank, uniformly in $\Theta$; when $m\geq n$, it means uniform full
column rank.
Consequently, if $y\in F_s$ has the representation in
\eqref{eq:face_signed_parameterization}, then
\begin{equation}
\label{eq:face_residual_representation}
    r(y)=C_{\Theta_s}\xi-c.
\end{equation}
A matrix $M\in\mathbb R^{n\times n}$ is a \emph{P-matrix} if all its principal
minors are positive. Equivalently, every nonzero $p\in\mathbb R^n$ has an
index $i$ such that $p_i(Mp)_i>0$; see \cite{cottle1992linear}.

Let $\Pi_{\mathcal C}$ denote the Euclidean projection onto $\mathcal C$. Although $\mathcal C$ is nonconvex, it is the Cartesian
product of $n$ unions of two nonnegative coordinate axes. Hence, its Euclidean
projection is available in closed form.

\begin{proposition}
For $(\xi,\eta)\in\mathbb R^n\times\mathbb R^n$, a projection onto
$\mathcal C$ is obtained coordinatewise by setting
\begin{equation}\label{eq:scalar_projection_rule}
    (u_i,v_i)=
    \begin{cases}
        (\xi_i^+,0),&\xi_i^+\geq\eta_i^+,\\[1mm]
        (0,\eta_i^+),&\eta_i^+>\xi_i^+,
    \end{cases}
    \qquad i=1,\ldots,n.
\end{equation}
Then $(u,v)\in\Pi_{\mathcal C}(\xi,\eta)$.
\end{proposition}
\begin{proof}
For each $i$, the closest points on the two nonnegative axes are
$(\xi_i^+,0)$ and $(0,\eta_i^+)$. Their squared distances differ by
\[
    \bigl((\xi_i-\xi_i^+)^2+\eta_i^2\bigr)
    -\bigl(\xi_i^2+(\eta_i-\eta_i^+)^2\bigr)
    =(\eta_i^+)^2-(\xi_i^+)^2.
\]
Thus the first candidate is no farther precisely when
$\xi_i^+\geq\eta_i^+$, proving \eqref{eq:scalar_projection_rule}.
\end{proof}
When $\xi_i^+=\eta_i^+$, \eqref{eq:scalar_projection_rule} selects
$(\xi_i^+,0)$. If their common value is positive, then $(0,\eta_i^+)$ is
another Euclidean projection, so $\Pi_{\mathcal C}$ is set-valued at this
coordinate. When the choice of projection matters, the analysis retains the
full set-valued mapping.
\begin{definition}
For $\alpha>0$, a point $y\in\mathcal C$ is \emph{$\alpha$-projected
stationary} if
\begin{equation}\label{eq:projected_stationary}
    y\in\Pi_{\mathcal C}\bigl(y-\alpha\nabla\Phi(y)\bigr).
\end{equation}
\end{definition}
The following condition states that projected stationarity is sufficient for
global optimality in \eqref{eq:general_model}:
\begin{equation}
\label{eq:projected_stationarity_implies_solution}
    \bigcup_{\alpha>0}
    \left\{y\in\mathcal C:
    y\in\Pi_{\mathcal C}\bigl(y-\alpha\nabla\Phi(y)\bigr)\right\}
    \subseteq\mathcal S.
\end{equation}
As this is not always true, later we will discuss in which situations this holds.
For a selected projection
$\hat y\in\Pi_{\mathcal C}(y-\alpha\nabla\Phi(y))$, define the
projected-gradient residual by
\begin{equation}\label{eq:pg_residual}
    G_\alpha(y;\hat y):=\frac{1}{\alpha}(y-\hat y).
\end{equation}
The dependence on $\hat y$ accounts for the possible nonuniqueness of the
projection.
\section{Algorithmic Framework}\label{sec:algorithmic_framework}

Our method applies a projected-gradient step to \eqref{eq:compact_model} and
solves the linear system associated with the selected face
$F(\bar y^{k+1})$. The projected-gradient step produces the trial point
\(\bar y^{k+1}\); a refinement is accepted only if it belongs to
$\mathcal C$ and does not increase $\Phi$ relative to \(\bar y^{k+1}\).

Let \(y:=[u^\top,v^\top]^\top\in\mathbb R^{2n}\) and $H:=[P,~Q]$.
Problem \eqref{eq:general_model} can be written as
\begin{equation}\label{eq:compact_model}
    \min_{y\in\mathcal C}
    \Phi(y)
    :=
    \frac12\|Hy-c\|^2.
\end{equation}
Its gradient is \(\nabla\Phi(y)=H^\top(Hy-c)\), which is Lipschitz continuous
with constant \(L=\|H\|_2^2\).
Starting from $y^k\in\mathcal C$, we compute
\begin{equation}\label{eq:pg_candidate}
    \bar y^{k+1}
    \in
    \Pi_{\mathcal C}
    \left(
        y^k-\alpha_k\nabla\Phi(y^k)
    \right).
\end{equation}

For every iteration, denote $\vartheta_k:=\vartheta(y^k)$, and 
$\bar\vartheta_{k+1}:=\vartheta(\bar y^{k+1})$.
Define the diagonal selector associated with the deterministic face
$F(\bar y^{k+1})$ by
\[
    \bar\Theta_{k+1}
    :=
    \frac12\left(
        I+\operatorname{diag}(\chi(\bar y^{k+1}))
    \right).
\]
The projected-gradient trial point serves as the reference point in the
acceptance test. The linear-system refinement associated with its selected
face is based on
\begin{equation}
\label{eq:selected_face_refinement_system}
    C_{\bar\Theta_{k+1}}\widetilde\xi=c.
\end{equation}
If $(2\bar\Theta_{k+1}-I)\widetilde\xi\geq0$, the corresponding point on the
selected face is
\begin{equation}
\label{eq:selected_face_refinement_point}
    \widetilde y
    :=
    \bigl(
        \bar\Theta_{k+1}\widetilde\xi,
        -(I-\bar\Theta_{k+1})\widetilde\xi
    \bigr).
\end{equation}
If the system cannot be solved or no
feasible trial point is produced, no refinement is accepted. 
A refinement is attempted
only when a prescribed trigger $\mathcal T_k$ holds. Initialize $\nu_0:=0$
and set
\begin{equation}
\label{eq:activity_vector_stability_counter}
    \nu_{k+1}:=
    \begin{cases}
        \nu_k+1,&\bar\vartheta_{k+1}=\vartheta_k,\\
        0,&\text{otherwise}.
    \end{cases}
\end{equation}
For prescribed $q\geq1$ and $\tau\geq0$, define the trigger
\begin{equation}\label{eq:refinement_trigger}
    \mathcal T_k\;{\rm holds.}
    \quad\Longleftrightarrow\quad
    \{\rho(\bar y^{k+1})\leq\tau
    \;\text{or}\;
    \nu_{k+1}\geq q\},
\end{equation}
where
\begin{equation}
\label{eq:relative_residual_scaling}
    \rho(y):=\frac{\|Hy-c\|}{s_c},
    \qquad
    s_c:=\max\{1,\|c\|\}.
\end{equation}
The first condition of \eqref{eq:refinement_trigger} tests the normalized equation residual, while the second
records that the activity vector $\vartheta$ has
remained unchanged for at least \(q\) consecutive iterations.
When $\mathcal T_k$ holds, let
$\widehat y^{k+1}$ be a refinement obtained from
\eqref{eq:selected_face_refinement_system} and
accept it according to
\begin{equation}\label{eq:refinement_safeguard}
    y^{k+1}
    =
    \begin{cases}
        \widehat y^{k+1},
        &
        \widehat y^{k+1}\in\mathcal C
        \ \text{and}\
        \Phi(\widehat y^{k+1})
        \leq
        \Phi(\bar y^{k+1}),
        \\[1mm]
        \bar y^{k+1},
        &
        \text{otherwise}.
    \end{cases}
\end{equation}
If no refinement is performed, we simply set
$y^{k+1}=\bar y^{k+1}$. Thus, irrespective of how the refinement candidate
is constructed,
\begin{equation}\label{eq:refinement_no_increase}
    \Phi(y^{k+1})
    \leq
    \Phi(\bar y^{k+1}).
\end{equation}

Algorithm~\ref{alg:main} summarizes the framework. The construction of
$\widehat y^{k+1}$ is deliberately left open here; concrete GAVE and LCP
refinements are specified in Section~\ref{sec:numerics}.

\begin{algorithm}[htbp]
\caption{Projected-gradient method with safeguarded refinement}
\label{alg:main}
\begin{algorithmic}[1]
\REQUIRE Problem data $(P,Q,c)$, a refinement associated with
\eqref{eq:selected_face_refinement_system}, trigger parameters $q\geq1$ and
$\tau\geq0$, and $y^0\in\mathcal C$.
\STATE Initialize $\nu_0=0$.
\FOR{$k=0,1,2,\ldots$}
    \STATE Select a stepsize $\alpha_k>0$.
    \STATE Compute $\bar y^{k+1}$ from \eqref{eq:pg_candidate}.
    \IF{the prescribed trigger $\mathcal T_k$ holds}
        \STATE Compute a refinement candidate $\widehat y^{k+1}$ using
        \eqref{eq:selected_face_refinement_system}.
    \ELSE
        \STATE Set $\widehat y^{k+1}=\bar y^{k+1}$.
    \ENDIF
    \STATE Determine $y^{k+1}$ according to
    \eqref{eq:refinement_safeguard}.
\ENDFOR
\end{algorithmic}
\end{algorithm}

\section{Convergence Analysis}\label{sec:convergence}

Using the descent property, we first establish worst-case complexity
for every refinement satisfying \eqref{eq:refinement_safeguard}. We then use
active faces and the matrix family \eqref{eq:matrix_family} 
to characterize conditions under which
projected stationarity implies zero residual, including at biactive points.
With a concrete analysis, error bounds for the complementarity constrained 
least squares is derived, which yield linear convergence of the algorithm. 
Finally, we demonstrate explicit bounds for
active-face identification and finite-termination of the algorithm. 
The GAVE- and LCP-specific assumptions enter only through verification of the
common matrix conditions.

\subsection{Convergence of the algorithm}

As a start, we verify the
descent property. Once this is done, convergence to projected stationarity and worst-case complexity follow from
standard telescoping arguments. As usual, we use the following stepsize condition.

\begin{assumption}\label{ass:stepsize}
There exist constants $\underline\alpha$ and $\overline\alpha$ such that
\(0<\underline\alpha\leq\alpha_k\leq\overline\alpha<1/L\), for all $k$.
\end{assumption}

For later use, define the uniform descent constant and the initial objective
gap by
\[
    c_d:=\frac{1}{2\overline\alpha}-\frac{L}{2}>0,
    \qquad
    \Delta_0:=\Phi(y^0)-\inf_{y\in\mathcal C}\Phi(y).
\]

With this step size, the following lemma transfers the decrease of the
projected-gradient point to the safeguarded iterate. 
\begin{lemma}\label{lem:pg_descent}
Suppose that Assumption~\ref{ass:stepsize} holds. Then the iterate satisfies
\begin{equation}\label{eq:descent_final}
    \Phi(y^{k+1})
    \leq
    \Phi(y^k)
    -c_d\|\bar y^{k+1}-y^k\|^2.
\end{equation}
\end{lemma}

\begin{proof}
Since $y^k\in\mathcal C$ and $\bar y^{k+1}$ is a Euclidean projection,
\[
    \|\bar y^{k+1}-(y^k-\alpha_k\nabla\Phi(y^k))\|^2
    \leq
    \|y^k-(y^k-\alpha_k\nabla\Phi(y^k))\|^2.
\]
Expanding this inequality gives
\[
    \langle\nabla\Phi(y^k),\bar y^{k+1}-y^k\rangle
    \leq
    -\frac{1}{2\alpha_k}\|\bar y^{k+1}-y^k\|^2.
\]
 As $\Phi(\cdot)$ has an
$L$-Lipschitz gradient, we have
\[
    \Phi(\bar y^{k+1})
    \leq\Phi(y^k)+\langle\nabla\Phi(y^k),\bar y^{k+1}-y^k\rangle+\frac{L}{2}\|\bar y^{k+1}-y^k\|^2.
\]
Combining the above two inequalities proves 
\begin{equation}\label{eq:basic_descent}
    \Phi(\bar y^{k+1})
    \leq
    \Phi(y^k)
    -\left(\frac{1}{2\alpha_k}-\frac{L}{2}\right)
    \|\bar y^{k+1}-y^k\|^2.
\end{equation}
Under Assumption~\ref{ass:stepsize},
\[
    \frac{1}{2\alpha_k}-\frac{L}{2}
    \geq
    c_d>0.
\]
Combining this bound with \eqref{eq:basic_descent} and
\eqref{eq:refinement_no_increase} proves \eqref{eq:descent_final}.
\end{proof}
 
We now obtain the basic asymptotic result. The normalized projected
displacements are precisely the stationarity measure in
\eqref{eq:pg_residual}.
\begin{theorem}\label{thm:subseq_convergence}
Suppose that Assumption~\ref{ass:stepsize} holds. Then
\begin{equation}\label{eq:summable_steps}
    \sum_{k=0}^{\infty}
    \|\bar y^{k+1}-y^k\|^2<\infty,
    \qquad
    \|G_{\alpha_k}(y^k;\bar y^{k+1})\|\to0.
\end{equation}
Moreover, every accumulation point $y^*$ of $\{y^k\}$ is
$\alpha_*$-projected stationary for some
$\alpha_*\in[\underline\alpha,\overline\alpha]$.
\end{theorem}

\begin{proof}
By Lemma \ref{lem:pg_descent}, $\{\Phi(y^k)\}$ is nonincreasing and,
because $\Phi$ is nonnegative, convergent. Summing
\eqref{eq:descent_final} gives
\[
    c_d\sum_{k=0}^{K}\|\bar y^{k+1}-y^k\|^2
    \leq \Phi(y^0)-\Phi(y^{K+1})
    \leq \Delta_0.
\]
Letting $K\to\infty$ proves the first part of \eqref{eq:summable_steps}. Hence
$\|\bar y^{k+1}-y^k\|\to0$, and the lower stepsize bound yields
$\|G_{\alpha_k}(y^k;\bar y^{k+1})\|\to0$.

Let $y^{k_j}\to y^*$. After passing to a further subsequence,
$\alpha_{k_j}\to\alpha_*$ for some
$\alpha_*\in[\underline\alpha,\overline\alpha]$. As $\|\bar y^{k+1}-y^k\|\to0$, we also have
$\bar y^{k_j+1}\to y^*$, while \eqref{eq:pg_candidate} gives
\[
    \bar y^{k_j+1}\in
    \Pi_{\mathcal C}
    \bigl(y^{k_j}-\alpha_{k_j}\nabla\Phi(y^{k_j})\bigr).
\]
As the projection mapping onto a closed set has a closed graph, passing to
the limit gives \eqref{eq:projected_stationary}. 
\end{proof}

With the decrease property, we can also quantify the number of iterations
needed to produce an approximately projected stationary point.
\begin{theorem}\label{thm:complexity}
Suppose that Assumption~\ref{ass:stepsize} holds. For any $\epsilon>0$, define
\(
    T_\epsilon
    :=\min\bigl\{k\geq0:
    \|G_{\alpha_k}(y^k;\bar y^{k+1})\|\leq\epsilon\bigr\},
\) 
with the convention $\min\varnothing:=+\infty$. Then $T_\epsilon$ is finite
and
\begin{equation}\label{eq:complexity_bound}
    T_\epsilon
    \leq
    \frac{\Delta_0}{c_d\underline\alpha^2\epsilon^2}.
\end{equation}
Consequently, the projected-gradient residual is at most
$\epsilon$ in $O(\epsilon^{-2})$ iterations.
\end{theorem}

\begin{proof}
Using \eqref{eq:pg_residual} in \eqref{eq:descent_final} and
$\alpha_k\geq\underline\alpha$ yields
\[
    c_d\underline\alpha^2
    \|G_{\alpha_k}(y^k;\bar y^{k+1})\|^2
    \leq
    \Phi(y^k)-\Phi(y^{k+1}).
\]
If $T_\epsilon=+\infty$, then the residual norm is greater than $\epsilon$
at every iteration. Summing the one-step estimate over $k=0,\ldots,N$ would
then give
\[
    c_d\underline\alpha^2(N+1)\epsilon^2
    < \Delta_0
\]
for any $N$, which is impossible when $N$ is large. Hence $T_\epsilon$ is finite. If
$T_\epsilon=0$, then \eqref{eq:complexity_bound} is immediate. Otherwise,
its definition gives
\[
    \|G_{\alpha_k}(y^k;\bar y^{k+1})\|>\epsilon,
    \qquad k=0,\ldots,T_\epsilon-1.
\]
Summing the one-step estimate over these indices yields
\[
    c_d\underline\alpha^2 T_\epsilon\epsilon^2
    < \Phi(y^0)-\Phi(y^{T_\epsilon})
    \leq \Delta_0,
\]
which proves \eqref{eq:complexity_bound} and the stated
$O(\epsilon^{-2})$ bound.
\end{proof}

The bound in Theorem~\ref{thm:complexity} concerns approximate projected
stationarity of the baseline projected-gradient step. By itself, it does not
give an $\epsilon$-bound for the equation residual $\|Hy^k-c\|$, the distance
to the solution set, or the objective gap. Such conclusions require the
conditions ensuring that projected stationary points belong to $\mathcal S$
or the error bounds developed below.

\subsection{Projected stationarity and global convergence}

Projected stationarity need not imply zero residual on the nonconvex set
$\mathcal C$. At a strictly complementary point, stationarity reduces to a
normal equation on one active face. At a biactive coordinate, however, more
than one face is admissible, so a fixed active-face matrix is insufficient.
We begin with an auxiliary facewise characterization and then address the
biactive case through the matrix family \eqref{eq:matrix_family}.

\begin{proposition}
\label{prop:active_face}
Let $y^*=(u^*,v^*)$ be an $\alpha$-projected stationary point of
\eqref{eq:compact_model}. Then
\begin{align}
    P_{:i}^{\top}r(y^*)&=0,
    &&i\in\mathcal I_{+0}(y^*),\label{eq:active_u_normal}\\
    Q_{:i}^{\top}r(y^*)&=0,
    &&i\in\mathcal I_{0+}(y^*),\label{eq:active_v_normal}\\
    P_{:i}^{\top}r(y^*)&\geq0,
    \quad Q_{:i}^{\top}r(y^*)\geq0,
    &&i\in\mathcal I_{00}(y^*).\label{eq:biactive_gradient}
\end{align}
If $\mathcal I_{00}(y^*)=\emptyset$, then
\begin{equation}\label{eq:active_face_normal_equation}
    \mathcal H(y^*)^{\top}r(y^*)=0.
\end{equation}
Moreover,
\[
    y^*\in
    \underset{y\in F(y^*)}{\operatorname{argmin}}\,\Phi(y).
\]
\end{proposition}

\begin{proof}
 As $y^*$ is projected stationary, we have  $y^*\in\Pi_{\mathcal C}\bigl(y^*-\alpha\nabla\Phi(y^*)\bigr)$. When
$i\in\mathcal I_{+0}(y^*)$, the $i$th projected pair of $y^*-\alpha\nabla\Phi(y^*)$ should equal $(u_i^*,0)$ some $u_i^*>0$. Write $\nabla\Phi(y^*)=[P^{\top}r(y^*);Q^{\top}r(y^*)]$. 
The closest point on the first nonnegative axis to $y^*-\alpha\nabla\Phi(y^*)$ has first component
\(
    \bigl(u_i^*-\alpha P_{:i}^{\top}r(y^*)\bigr)^+.
\)
It can equal $u_i^*$ only if $P_{:i}^{\top}r(y^*)=0$, so 
\eqref{eq:active_u_normal} is proved. The proof of \eqref{eq:active_v_normal} is
identical. If $i\in\mathcal I_{00}(y^*)$, the projection of
$(-\alpha P_{:i}^{\top}r(y^*),-\alpha Q_{:i}^{\top}r(y^*))$ is $(0,0)$.
The projection rule \eqref{eq:scalar_projection_rule} therefore proves
\eqref{eq:biactive_gradient}. 
Under strict complementarity, that is, $\mathcal I_{00}(y^*)=\emptyset$, combining \eqref{eq:active_u_normal} and \eqref{eq:active_v_normal} 
gives \eqref{eq:active_face_normal_equation}. For any
$\widetilde y=(\widetilde u,\widetilde v)\in F(y^*)$, set
$s^*:=u^*+v^*$ and $\widetilde s:=\widetilde u+\widetilde v$. By the definition of $\mathcal H(y^*)$, we have
\[
    r(\widetilde y)
    =r(y^*)+\mathcal H(y^*)(\widetilde s-s^*),
\]
and hence
\[
    \Phi(\widetilde y)-\Phi(y^*)
    =\left\langle
      \mathcal H(y^*)^\top r(y^*),\widetilde s-s^*
     \right\rangle
     +\frac12\left\|\mathcal H(y^*)(\widetilde s-s^*)\right\|^2
    \geq0
\]
by \eqref{eq:active_face_normal_equation}, which proves the facewise
optimality statement.
\end{proof}

The proposition is only facewise: a stationary point minimizes the quadratic
model on its active face, but the minimum residual on that face may be
positive. This active-face viewpoint is analogous to the active-set
decomposition used in MPCC analysis \cite[Section~4]{nurkanovic2026sqpcc}.

The active-face normal equation admits the following formulation in terms of
the matrix family \eqref{eq:matrix_family}.

\begin{corollary}
\label{cor:strict_complementarity_zero_residual}
Let $y^*$ be a strictly complementary $\alpha$-projected stationary point of
\eqref{eq:compact_model}, and define
\(
    \Theta_*:=\tfrac12
    \bigl(I+\operatorname{diag}(\vartheta(y^*))\bigr).
\)
Then
\begin{equation}
\label{eq:active_face_matrix_normal_equation}
    C_{\Theta_*}^{\top}r(y^*)=0.
\end{equation}
Moreover, $r(y^*)=0$ if
\begin{equation}
\label{eq:pointwise_matrix_conditions}
    \operatorname{rank}(C_{\Theta_*})=m
    \quad\text{or}\quad
    c\in\operatorname{Range}(C_{\Theta_*}).
\end{equation}
\end{corollary}

\begin{proof}
Set $D_*:=2\Theta_*-I=\operatorname{diag}(\vartheta(y^*))$. Strict
complementarity implies that $\Theta_*$ is a diagonal matrix with the $i$-th entry equal to $1$ for $i\in \mathcal I_{+0}$ and equal to $0$ for $i\in \mathcal I_{0+}$,
while  $D_*$ is nonsingular. 
Since $C_{\Theta_*}=P\Theta_*-Q(I-\Theta_*)$, we have
$\mathcal {H}(y^*)=C_{\Theta_*}D_*$. Hence
\eqref{eq:active_face_normal_equation} is equivalent to
\eqref{eq:active_face_matrix_normal_equation} in this situation. If
$\operatorname{rank}(C_{\Theta_*})=m$, then
$C_{\Theta_*}^{\top}$ is injective, so \eqref{eq:active_face_matrix_normal_equation}
gives $r(y^*)=0$. Alternatively,
\eqref{eq:face_residual_representation} and
$c\in\operatorname{Range}(C_{\Theta_*})$ give
$r(y^*)\in\operatorname{Range}(C_{\Theta_*})$, whereas
\eqref{eq:active_face_matrix_normal_equation} gives
$r(y^*)\in\operatorname{Range}(C_{\Theta_*})^\perp$. Thus $r(y^*)=0$.
\end{proof}

\begin{remark}
For a strictly complementary GAVE point, let
$D_*:=\operatorname{diag}(\vartheta(y^*))$. Then
$C_{\Theta_*}=A-BD_*$, so either condition in
\eqref{eq:pointwise_matrix_conditions} can be checked directly from
$A-BD_*$. For a strictly complementary LCP point, let
$\mathcal I:=\mathcal I_{+0}(y^*)$. After applying the same ordering
$(\mathcal I,\mathcal I^c)$ to its rows and columns, $C_{\Theta_*}$ has the
block form
\[
    \begin{bmatrix}
        M_{\mathcal I\mathcal I}&0\\
        M_{\mathcal I^c\mathcal I}&I
    \end{bmatrix}.
\]
Hence $\det(M_{\mathcal I\mathcal I})\neq0$ is sufficient for
$r(y^*)=0$.
\end{remark}

The preceding corollary gives conditions under which an individual strictly
complementary projected stationary point has zero residual. The following
theorem uses these pointwise conditions to establish convergence to the
solution set of \eqref{eq:compact_model}.

\begin{theorem} 
\label{thm:active_face_global_convergence}
Suppose that Assumption~\ref{ass:stepsize} holds, $\{y^k\}$ is bounded, and
every accumulation point $y^*$ satisfies $\mathcal I_{00}(y^*)=\emptyset$
and, with
\(\Theta_*:=\tfrac12(I+\operatorname{diag}(\vartheta(y^*)))\), either
$\operatorname{rank}(C_{\Theta_*})=m$ or
$c\in\operatorname{Range}(C_{\Theta_*})$. Then
$\mathcal S\neq\varnothing$, every accumulation point belongs to
$\mathcal S$, and $\operatorname{dist}(y^k,\mathcal S)\to0$.
Specifically, if $\mathcal S=\{y^\dagger\}$, then $y^k\to y^\dagger$.
\end{theorem}

\begin{proof}
Let $y^*$ be any accumulation point. Theorem~\ref{thm:subseq_convergence} shows that $y^*$ is
$\alpha_*$-projected stationary for some
$\alpha_*\in[\underline\alpha,\overline\alpha]$. Corollary~
\ref{cor:strict_complementarity_zero_residual} then gives
$r(y^*)=0$, so $y^*\in\mathcal S$, indicating that $\mathcal S \neq \emptyset$.

Since $\{y^k\}$ is bounded,
monotonicity of $\{\Phi(y^k)\}$ and continuity along a convergent subsequence
give $\Phi(y^k)\to0$. Moreover, a bounded sequence whose accumulation points
all lie in the closed set $\mathcal S$ satisfies
$\operatorname{dist}(y^k,\mathcal S)\to0$ by the standard subsequence
argument. If $\mathcal S=\{y^\dagger\}$, every accumulation point equals
$y^\dagger$, and boundedness then gives $y^k\to y^\dagger$.
\end{proof}

The theorem therefore upgrades a pointwise zero-residual criterion to
convergence of the generated sequence to the solution set. Strict
complementarity is used only to select a single active-face matrix at each
accumulation point.

If the zero-residual solution set $\mathcal S$ is empty, then the assumptions
of Theorem~\ref{thm:active_face_global_convergence} cannot all hold.
Nevertheless, problem~\eqref{eq:compact_model} still has a global minimizer. 
Indeed,
$\mathcal C$ is a finite union of polyhedral cones, and the image of
each such cone under $H$ is closed and polyhedral. Hence $H(\mathcal C)$ is
closed, so the distance from $c$ to $H(\mathcal C)$ is attained. However, 
Proposition~\ref{prop:active_face} only 
shows that a projected stationary point minimizes
$\Phi$ over its active face, it does not 
guarantee global optimality over $\mathcal C$. In this inconsistent case,
convergence to a global least-squares solution 
would require an additional condition, such as an appropriate projected
Polyak--\L{}ojasiewicz inequality. This case is beyond the scope of the paper. 

The preceding condition is pointwise and requires strict complementarity. At
a biactive coordinate, the key step is to construct a diagonal selector
$\Theta\in\mathcal D_{[0,1]}$ from the coordinatewise stationarity
inequalities. This gives one equation $C_\Theta^\top r(y)=0$ without selecting
one of the incident faces.

\begin{theorem}
\label{thm:uniform_matrix_stationarity}
Suppose that $m\leq n$ and $\mu_C>0$.
Then every $\alpha$-projected stationary point of
\eqref{eq:compact_model} belongs to $\mathcal S$, for every $\alpha>0$.
If, in addition, Assumption~\ref{ass:stepsize} holds and $\{y^k\}$ is
bounded, then
\begin{equation}
\label{eq:uniform_matrix_global_convergence}
    \Phi(y^k)\to0,
    \qquad
    \operatorname{dist}(y^k,\mathcal S)\to0,
\end{equation}
and every accumulation point belongs to $\mathcal S$.
\end{theorem}

\begin{proof}
Let $y=(u,v)$ be an $\alpha$-projected stationary point and write
$p_i:=P_{:i}^{\top}r(y)$ and $q_i:=Q_{:i}^{\top}r(y)$. Define
$\theta_i=1$ on $\mathcal I_{+0}(y)$ and $\theta_i=0$ on
$\mathcal I_{0+}(y)$. At a biactive index, where $p_i,q_i\geq0$ by
\eqref{eq:biactive_gradient}, set
\(
    \theta_i:=q_i/(p_i+q_i)
\)
if $p_i+q_i>0$, and set $\theta_i:=0$ otherwise. Then
$\Theta:=\operatorname{diag}(\theta)\in\mathcal D_{[0,1]}$ and
\(
    \theta_i p_i-(1-\theta_i)q_i=0
\)
for every $i$. Hence $C_\Theta^{\top}r(y)=0$.
The condition $\mu_C>0$ makes
$C_\Theta^{\top}$ injective, so $r(y)=0$.

Now assume that Assumption~\ref{ass:stepsize} holds and $\{y^k\}$ is
bounded. Theorem~\ref{thm:subseq_convergence} and the first conclusion show
that every accumulation point belongs to $\mathcal S$. As $\Phi$ is continuous and decreasing, we have $\Phi(y^k)\to0$. By the same compactness argument used
in the proof of Theorem~\ref{thm:active_face_global_convergence} we also have 
$\operatorname{dist}(y^k,\mathcal S)\to0$.
\end{proof}

The uniform matrix condition thus removes the strict-complementarity
requirement, as it makes $C_\Theta^\top$ injective for the
diagonal selector produced at a biactive point.

\begin{corollary}
\label{cor:gave_uniform_matrix_consequences}
For the GAVE \eqref{eq:gave}, suppose that $m\leq n$ and
\begin{equation}
\label{eq:gave_uniform_row_regular_condition}
    \inf_{D\in\mathcal D_{[-1,1]}}
    \sigma_{\min}(A-BD)>0.
\end{equation}
Then the conclusions of Theorem~\ref{thm:uniform_matrix_stationarity} hold.
If, more specifically, $m=n$ and
\begin{equation}
\label{eq:gave_uniform_regular_condition}
    \delta:=\sigma_{\min}(A)-\|B\|_2>0,
\end{equation}
then \eqref{eq:gave} has a unique solution $x^\dagger$. In this situation, every projected stationary point
equals
$y^\dagger:=((x^\dagger)^+,(x^\dagger)^-)$. Moreover, if Assumption~\ref{ass:stepsize} holds, then
\begin{equation}
\label{eq:gave_global_convergence}
    y^k\to y^\dagger,
    \qquad
    \Phi(y^k)\to0.
\end{equation}
\end{corollary}

\begin{proof}
For the GAVE representation, \eqref{eq:matrix_family} gives
\(C_\Theta=A-BD\), where $D:=2\Theta-I$ ranges over
$\mathcal D_{[-1,1]}$. Thus
\eqref{eq:gave_uniform_row_regular_condition} is precisely
the condition $\mu_C>0$.

Under \eqref{eq:gave_uniform_regular_condition},
$\sigma_{\min}(A-BD)\geq\delta$ for every
$D\in\mathcal D_{[-1,1]}$. Moreover,
$\|A^{-1}B\|_2<1$, so $x\mapsto A^{-1}(B|x|+b)$ is a contraction and has a
unique fixed point $x^\dagger$. For $y=(u,v)\in\mathcal C$, let $x:=u-v$.
Then $|x|=u+v$, $\|x\|=\|y\|$, and
\(
    \sqrt{2\Phi(y)}\geq\delta\|y\|-\|b\|.
\)
Hence $\Phi$ is coercive on $\mathcal C$. Following a similar argument to that in
Theorem~\ref{thm:uniform_matrix_stationarity}, unique solvability gives
\eqref{eq:gave_global_convergence}.
\end{proof}

\begin{remark}
The continuous family $\Theta\in\mathcal D_{[0,1]}$ accounts for the
fractional diagonal entries that may arise at biactive coordinates.
The condition $\mu_C>0$ is sufficient for
\eqref{eq:projected_stationarity_implies_solution} and is not necessary in
general. Notably, the solvability result of GAVE recovers that of \cite{xie2025randomized}.
\end{remark}

\begin{corollary}
\label{cor:lcp_p_matrix_consequences}
Suppose that $M$ is a P-matrix in the LCP \eqref{eq:lcp}. Then
\begin{equation}
\label{eq:lcp_uniform_matrix_modulus}
    \mu_{\rm L}
    :=
    \min_{\Theta\in\mathcal D_{[0,1]}}
    \sigma_{\min}(I-\Theta+M\Theta)
    >0.
\end{equation}
Consequently, every $\alpha$-projected stationary point of
\eqref{eq:compact_model} is the unique LCP solution
$y^\dagger=(z^\dagger,w^\dagger)$. If Assumption~\ref{ass:stepsize} holds,
then
\begin{equation}
\label{eq:lcp_global_convergence}
    (z^k,w^k)\to(z^\dagger,w^\dagger),
    \qquad
    \Phi(z^k,w^k)\to0.
\end{equation}
\end{corollary}

\begin{proof}
For the LCP representation, \eqref{eq:matrix_family} gives
\(C_\Theta=I-\Theta+M\Theta\). Fix
$\Theta=\operatorname{diag}(d)\in\mathcal D_{[0,1]}$. To prove that
$C_\Theta$ is nonsingular, suppose that $C_\Theta p=0$ and set
$v:=\Theta p$. If $v=0$, then $(I-\Theta)p=0$ and hence $p=0$. Otherwise,
$C_\Theta p=0$ gives $Mv=-(I-\Theta)p$, so
\[
    v_i(Mv)_i=-d_i(1-d_i)p_i^2\leq0,
    \qquad \forall i.
\]
This contradicts the P-matrix characterization. Therefore $C_\Theta$ is
nonsingular for every $\Theta\in\mathcal D_{[0,1]}$. Continuity of
$\Theta\mapsto\sigma_{\min}(C_\Theta)$ and compactness of
$\mathcal D_{[0,1]}$ give \eqref{eq:lcp_uniform_matrix_modulus}. Theorem~
\ref{thm:uniform_matrix_stationarity} applies, and the classical P-matrix
theory gives a unique LCP solution; see \cite{cottle1992linear}.

For boundedness, let $\mathcal E$ be the finite set of diagonal matrices with
diagonal entries in $\{0,1\}$. Equation
\eqref{eq:lcp_uniform_matrix_modulus} gives
\(\widehat\mu_{\rm L}:=\min_{E\in\mathcal E}
\sigma_{\min}(I-E+ME)>0\).
For every $(z,w)\in\mathcal C$, taking $x=z-w$, and selecting a binary matrix $E$ having $E_{ii}$ equal to $1$ if $x_i\geq0$, while $E_{ii}$ equal to $0$ if $x_i<0$, we have
\[
    Mz-w=(I-E+ME)x,
    \qquad \|x\|=\|(z,w)\|.
\]
Therefore $\sqrt{2\Phi(z,w)}=\|Mz-w-b\|\geq \widehat\mu_{\rm L}\|(z,w)\|-\|b\|$.

If Assumption~\ref{ass:stepsize} holds, the generated sequence
lies in a bounded level set by Lemma~\ref{lem:pg_descent} and the
coercivity estimate. Theorem~\ref{thm:uniform_matrix_stationarity} and
uniqueness show that the entire sequence converges to
$(z^\dagger,w^\dagger)$, and continuity of $\Phi$ gives
\eqref{eq:lcp_global_convergence}.
\end{proof}

\subsection{Error bounds and linear convergence}
\label{subsec:error_bound_linear_convergence}

The preceding results give qualitative convergence to the zero-residual
solution set. Next we show the linear convergence rate. A linear rate requires two quantitative links: the equation
residual must control distance to $\mathcal S$, and the projected-gradient
residual must control the same distance. The second link is more delicate
because the projection can be set-valued and the stepsize may vary. To proceed, we need the following assumption:

\begin{assumption}
\label{ass:bounded_initial_level_set}
The initial sublevel set 
\[
\mathcal L_0:=\{y\in\mathcal C:\Phi(y)\leq\Phi(y^0)\}
\]
is bounded.
\end{assumption}
By sufficient descent, the iterate sequence remains in \(\mathcal L_0\).
Thus, Assumption~\ref{ass:bounded_initial_level_set} concerns only the region
visited by the algorithm, rather than the whole feasible set \(\mathcal C\).

The following proposition supplies the residual error bound. Part~(i) uses
standard Hoffman bounds on finitely many faces; part~(ii) uses the common
matrix family to obtain a global estimate.

\begin{proposition}
\label{prop:problem_residual_error_bounds}
Suppose that \(\mathcal S\neq\varnothing\). Then the following statements hold.
\begin{enumerate}
    \item[(i)] If Assumption~\ref{ass:bounded_initial_level_set} holds, then
    there exists \(\kappa_r>0\) such that
    \begin{equation}
    \label{eq:abstract_residual_error_bound}
        \operatorname{dist}(y,\mathcal S)
        \leq
        \kappa_r\|Hy-c\|
        =
        \kappa_r\sqrt{2\Phi(y)},
        \qquad \forall y\in\mathcal L_0.
    \end{equation}

    \item[(ii)] If \(m\geq n\) and \(\mu_C>0\), then
    \(\mathcal S=\{y^\dagger\}\), and
    \begin{equation}
    \label{eq:general_global_error_bound}
        \|y-y^\dagger\|
        \leq
        \frac{1}{\mu_C}\|Pu+Qv-c\|,
        \qquad
        \forall y=(u,v)\in\mathcal C.
    \end{equation}
\end{enumerate}
\end{proposition}

\begin{proof}
For each $F\in\mathfrak F$, set
\(\mathcal S_F:=\{y\in F:Hy-c=0\}\). If
\(\mathcal S_F\neq\varnothing\), Hoffman's error bound
\cite{hoffman1952approximate} guarantees a
constant \(\kappa_F>0\) such that
\[
    \operatorname{dist}(y,\mathcal S)
    \leq
    \operatorname{dist}(y,\mathcal S_F)
    \leq
    \kappa_F\|Hy-c\|,
    \qquad y\in F.
\]
If \(\mathcal S_F=\varnothing\), set
\(\mathcal K_F:=F\cap\mathcal L_0\).
Whenever \(\mathcal K_F\neq\varnothing\), compactness gives
\(\epsilon_F:=\min_{y\in\mathcal K_F}\|Hy-c\|>0\) and
\(R_F:=\max_{y\in\mathcal K_F}
\operatorname{dist}(y,\mathcal S)<\infty\).
Consequently,
\[
    \operatorname{dist}(y,\mathcal S)
    \leq
    \frac{R_F}{\epsilon_F}\|Hy-c\|,
    \qquad y\in\mathcal K_F.
\]
Taking the maximum of these finitely many constants over all faces $F\in\mathfrak F$ proves part~(i).

For part~(ii), fix \(y=(u,v)\in\mathcal C\) and
\(y^\dagger=(u^\dagger,v^\dagger)\in\mathcal S\), and set
\(d:=(u-v)-(u^\dagger-v^\dagger)\). By complementarity, there exists some
\(\Theta\in\mathcal D_{[0,1]}\) such that
\(u-u^\dagger=\Theta d\) and
\(v-v^\dagger=-(I-\Theta)d\). As $Pu^\dagger+Qv^\dagger-c=0$, it follows that
\[
\begin{aligned}
    Pu+Qv-c
    &=
    P(u-u^\dagger)+Q(v-v^\dagger)\\
    &=
    \bigl(P\Theta-Q(I-\Theta)\bigr)d
    =C_\Theta d.
\end{aligned}
\]
Moreover, because every diagonal entry of \(\Theta\) lies in \([0,1]\),
\[
\begin{aligned}
    \|y-y^\dagger\|^2
    &=\|\Theta d\|^2+\|(I-\Theta)d\|^2
    \leq\|d\|^2.
\end{aligned}
\]
The definition of \(\mu_C\) now gives
\[
    \|Pu+Qv-c\|
    \geq
    \mu_C\|d\|
    \geq
    \mu_C\|y-y^\dagger\|,
\]
which proves \eqref{eq:general_global_error_bound}.
Applying this bound to any other point in \(\mathcal S\) proves uniqueness.
\end{proof}

The bounded-sublevel-set assumption in part~(i) controls faces that do not
intersect $\mathcal S$. The stronger matrix condition in part~(ii) removes
this localization and also implies uniqueness.

The next corollary specializes the global bound in part~(ii) to the two
problem classes considered in this paper.

\begin{corollary}
\label{cor:gave_lcp_residual_error_bounds}
The following statements hold.
\begin{enumerate}
    \item[(i)] Consider a GAVE with \(m=n\) and
    \(\mathcal S\neq\varnothing\). Define
    \[
        \mu_{\rm G}
        :=
        \inf_{D\in\mathcal D_{[-1,1]}}
        \sigma_{\min}(A-BD).
    \]
    If \(\mu_{\rm G}>0\), then \(\mathcal S=\{y^\dagger\}\), where
    \(y^\dagger=((x^\dagger)^+,(x^\dagger)^-)\), and
    \[
        \|y-y^\dagger\|
        \leq
        \frac{1}{\mu_{\rm G}}
        \|Ax-B|x|-b\|,
        \qquad
        \forall y=(x^+,x^-)\in\mathcal C.
    \]
    In particular, if
    \(\delta=\sigma_{\min}(A)-\|B\|_2>0\), then
    \(\mu_{\rm G}\geq\delta\).

    \item[(ii)] Consider an LCP with a P-matrix $M$, and let
    \(\mu_{\rm L}>0\) be defined by
    \eqref{eq:lcp_uniform_matrix_modulus}. Then
    \(\mathcal S=\{y^\dagger\}\), where
    \(y^\dagger=(z^\dagger,w^\dagger)\) is the unique LCP solution. Moreover,
    \[
        \|y-y^\dagger\|
        \leq
        \frac{1}{\mu_{\rm L}}
        \|Mz-w-b\|,
        \qquad
        \forall y=(z,w)\in\mathcal C.
    \]
\end{enumerate}
\end{corollary}

\begin{proof}
For the representation of \eqref{eq:gave} in \eqref{eq:general_model},
\(C_\Theta=A-B(2\Theta-I)\). As \(\Theta\)
ranges over \(\mathcal D_{[0,1]}\), \(D:=2\Theta-I\) ranges over
\(\mathcal D_{[-1,1]}\), so \(\mu_C=\mu_{\rm G}\). For the representation of
\eqref{eq:lcp} in \eqref{eq:general_model},
\(C_\Theta=I-\Theta+M\Theta\), and
\eqref{eq:lcp_uniform_matrix_modulus} gives the corresponding value of
\(\mu_C\).

The uniqueness statements and residual error bounds now follow directly from
part~(ii) of Proposition~\ref{prop:problem_residual_error_bounds} under these
two representations. Finally,
\(\sigma_{\min}(A-BD)\geq\sigma_{\min}(A)-\|B\|_2\) for every
\(D\in\mathcal D_{[-1,1]}\), which gives \(\mu_{\rm G}\geq\delta\).
\end{proof}

The preceding convergence results allow $\alpha_k$ to vary arbitrarily in
$[\underline\alpha,\overline\alpha]$. For the linear-rate analysis, however,
the projected-gradient error-bound constant may depend on the stepsize.
We therefore first establish the bound for each fixed $\alpha$ and then
obtain a uniform constant by restricting the stepsizes to a finite set
$\mathcal A$. This still permits variable stepsizes and avoids requiring
the error-bound constants to be uniformly bounded over the entire interval.

As specified in \eqref{eq:scalar_projection_rule}, the algorithm uses a
single-valued selection of $\Pi_{\mathcal C}$, choosing
$(\xi_i^+,0)$ whenever $\xi_i^+=\eta_i^+$. However, the underlying Euclidean
projection is nevertheless set-valued when $\xi_i^+=\eta_i^+>0$, because
both $(\xi_i^+,0)$ and $(0,\eta_i^+)$ are then distinct projections.
Define the set-valued projected-gradient mapping
\[
    \mathcal G_\alpha(y)
    :=
    \left\{
        \frac{1}{\alpha}(y-\widehat y):
        \widehat y\in
        \Pi_{\mathcal C}
        \bigl(y-\alpha\nabla\Phi(y)\bigr)
    \right\},
\]
and its least-norm residual
\(
    g_\alpha(y)
    :=
    \operatorname{dist}\bigl(0,\mathcal G_\alpha(y)\bigr).
\)

The projection can be represented by finitely
many affine mappings on polyhedral regions. Hoffman bounds then apply to each
piece, including pieces whose zero set is empty.

\begin{proposition}
\label{prop:pg_error_bound}
Suppose that \(\mathcal S\neq\varnothing\),
Assumption~\ref{ass:bounded_initial_level_set} holds, and
condition~\eqref{eq:projected_stationarity_implies_solution} holds.
For every fixed
\(\alpha\in[\underline\alpha,\overline\alpha]\), there exists
\(\kappa_\alpha>0\) such that
\begin{equation}
\label{eq:fixed_pg_error_bound}
    \operatorname{dist}(y,\mathcal S)
    \leq
    \kappa_\alpha g_\alpha(y),
    \qquad
    \forall y\in\mathcal L_0.
\end{equation}
Consequently, for every selected projection
\(\widehat y\in\Pi_{\mathcal C}
(y-\alpha\nabla\Phi(y))\),
\begin{equation}
\label{eq:fixed_selected_pg_error_bound}
    \operatorname{dist}(y,\mathcal S)
    \leq
    \frac{\kappa_\alpha}{\alpha}
    \|y-\widehat y\|,
    \qquad
    \forall y\in\mathcal L_0.
\end{equation}
\end{proposition}

\begin{proof}
Fix \(\alpha\in[\underline\alpha,\overline\alpha]\) and define the affine
map \(z_\alpha(y):=y-\alpha\nabla\Phi(y)\). Recall from
\eqref{eq:finite_face_collection} that
\(\mathcal C=\bigcup_{F\in\mathfrak F}F\), where every $F$ is a polyhedral
cone.
Let \(\mathcal E:=\{\operatorname{diag}(e):e\in\{0,1\}^n\}\). For
\(E_\xi,E_\eta\in\mathcal E\), define the sign region
\(
    \mathcal R_{E_\xi,E_\eta}
    :=
    \left\{(\xi,\eta)\in\mathbb R^{2n}:
    (2E_\xi-I)\xi\geq0,\quad
    (2E_\eta-I)\eta\geq0
    \right\}.
\)
For every \((\xi,\eta)\in\mathcal R_{E_\xi,E_\eta}\),
\(\xi^+=E_\xi\xi\) and \(\eta^+=E_\eta\eta\). Moreover,
\(\mathbb R^{2n}=\bigcup_{E_\xi,E_\eta\in\mathcal E}
\mathcal R_{E_\xi,E_\eta}\). Let \(s\in\{-1,1\}^n\).
Define
\[
\begin{aligned}
    \mathcal R_{E_\xi,E_\eta,s}
    :=\bigl\{(\xi,\eta)\in
    \mathcal R_{E_\xi,E_\eta}:\;&
    (E_\xi\xi)_i\geq(E_\eta\eta)_i
    &&\text{if }s_i=1,\\
    & (E_\eta\eta)_i\geq(E_\xi\xi)_i
    &&\text{if }s_i=-1,\quad i=1,\ldots,n
    \bigr\}.
\end{aligned}
\]
Then \(\mathcal R_{E_\xi,E_\eta}
=
\bigcup_{s\in\{-1,1\}^n}
\mathcal R_{E_\xi,E_\eta,s}.\) Define the linear
projection selection \(p_{E_\xi,E_\eta,s}\) coordinatewise by
\[
    \bigl(p_{E_\xi,E_\eta,s}(\xi,\eta)\bigr)_i
    :=
    \begin{cases}
        ((E_\xi\xi)_i,0),&s_i=1,\\
        (0,(E_\eta\eta)_i),&s_i=-1.
    \end{cases}
\]
If \((E_\xi\xi)_i=(E_\eta\eta)_i>0\), both choices of \(s_i\) satisfy the
defining inequalities and give the two distinct projections at coordinate
\(i\).

Let \(\mathcal J_\alpha\) index all tuples
\(j=(F,E_\xi,E_\eta,s)\), where $F\in\mathfrak F$. For each such tuple,
define
\(
    D_j
    :=
    F\cap
    z_\alpha^{-1}(\mathcal R_{E_\xi,E_\eta,s}),
\)
and
\(
    p_j(\xi,\eta):=p_{E_\xi,E_\eta,s}(\xi,\eta).
\)
Since \(z_\alpha\) is affine, every \(D_j\) is a closed polyhedral subset
of \(\mathcal C\). Moreover,
\[
    \mathcal C=\bigcup_{j\in\mathcal J_\alpha}D_j,
    \qquad
    \Pi_{\mathcal C}(z_\alpha(y))
    =
    \{p_j(z_\alpha(y)):j\in\mathcal J_\alpha,\ y\in D_j\}.
\]
For each \(j\in\mathcal J_\alpha\), define the affine mapping
\(G_j(y):=\alpha^{-1}(y-p_j(z_\alpha(y)))\), and set
\(\mathcal J_\alpha(y):=\{j\in\mathcal J_\alpha:y\in D_j\}\). By this definition, 
$p_j(z_\alpha(y))$ is one possible next iteration of $z_\alpha(y)$, while $\mathcal J_\alpha(y)$ 
collects all the possibilities. Then
\(
    \mathcal G_\alpha(y)
    =
    \{G_j(y):j\in\mathcal J_\alpha(y)\}.
\)
For every \(j\in\mathcal J_\alpha\), set
\(\mathcal S_j:=\{y\in D_j:G_j(y)=0\}\), which consists of all the stationary points in the area $D_j$.
For every $y\in\mathcal S$, we have $r(y)=0$ and hence
$\nabla\Phi(y)=H^\top r(y)=0$. Since $y\in\mathcal C$,
\(
    y=\Pi_{\mathcal C}(y)
    =
    \Pi_{\mathcal C}
    \bigl(y-\alpha\nabla\Phi(y)\bigr),
\)
and therefore $\mathcal G_\alpha(y)=\{0\}$.
Conversely, as we assume \eqref{eq:projected_stationarity_implies_solution} holds, 
every point in \(\mathcal G_\alpha^{-1}(0)\) belongs to \(\mathcal S\).
Hence
\[
    \mathcal G_\alpha^{-1}(0)
    =
    \bigcup_{j\in\mathcal J_\alpha}\mathcal S_j
    =
    \mathcal S.
\]

Consider  an index \(j\) for which \(\mathcal S_j\neq\varnothing\).
The set \(\mathcal S_j\) is described by the polyhedral constraints defining
\(D_j\), together with the affine equation \(G_j(y)=0\). Hoffman's error
bound \cite{hoffman1952approximate}  provides \(h_j>0\) such that,
for every \(y\in D_j\),
\begin{equation}
\label{eq:piecewise_hoffman_bound}
    \operatorname{dist}(y,\mathcal S)
    \leq
    \operatorname{dist}(y,\mathcal S_j)
    \leq
    h_j\|G_j(y)\|,
\end{equation}
where the first inequality follows from
\(\mathcal S_j\subseteq\mathcal S\).

By Assumption~\ref{ass:bounded_initial_level_set}, the closed level set
\(\mathcal L_0\) is compact, and 
\(\mathcal S\) is a nonempty closed subset of \(\mathcal L_0\). Set
\(R_0:=\max_{y\in\mathcal L_0}
\operatorname{dist}(y,\mathcal S)<\infty\). When \(D_j\cap\mathcal L_0\neq\varnothing\) and
\(\mathcal S_j=\varnothing\), we obtain
\(\delta_j:=\min_{y\in D_j\cap\mathcal L_0}\|G_j(y)\|>0\) , following the same compactness argument used for the case
\(\mathcal S_F=\varnothing\) in the proof of
Proposition~\ref{prop:problem_residual_error_bounds}.
Consequently, if $D_j$ doesn't contain any point in $S$,
\begin{equation}
\label{eq:empty_piece_error_bound}
    \operatorname{dist}(y,\mathcal S)
    \leq
    \frac{R_0}{\delta_j}\|G_j(y)\|,
    \qquad y\in D_j\cap\mathcal L_0.
\end{equation}

Combining these two cases, we define the index set
\(
    \mathcal J_\alpha^0
    :=
    \{j\in\mathcal J_\alpha:
    D_j\cap\mathcal L_0\neq\varnothing\},
\)
which is finite and nonempty. For every \(j\in\mathcal J_\alpha^0\), set
\(
    c_j
    :=
    \begin{cases}
        h_j,&\mathcal S_j\neq\varnothing,\\
        R_0/\delta_j,&\mathcal S_j=\varnothing.
    \end{cases} 
\), and 
\(
    \kappa_\alpha
    :=
    \max_{j\in\mathcal J_\alpha^0}c_j.
\)
For any \(y\in\mathcal L_0\), choose
\(
    j_y
    \in
    \underset{j\in\mathcal J_\alpha(y)}{\operatorname{argmin}}
    \|G_j(y)\|.
\)
Then \(j_y\in\mathcal J_\alpha^0\) and
\(g_\alpha(y)=\|G_{j_y}(y)\|\). Applying
\eqref{eq:piecewise_hoffman_bound} or
\eqref{eq:empty_piece_error_bound}, gives
\[
    \operatorname{dist}(y,\mathcal S)
    \leq
    c_{j_y}\|G_{j_y}(y)\|
    \leq
    \kappa_\alpha g_\alpha(y),
\]
which proves \eqref{eq:fixed_pg_error_bound}.

Finally, for any
\(\widehat y\in\Pi_{\mathcal C}(y-\alpha\nabla\Phi(y))\), the representation
of \(\mathcal G_\alpha(y)\) gives an index
\(j\in\mathcal J_\alpha(y)\) such that
\(\alpha^{-1}(y-\widehat y)=G_j(y)\). Therefore
\[ 
    \operatorname{dist}(y,\mathcal S)
    \leq
    \kappa_\alpha g_\alpha(y)
    \leq
    \kappa_\alpha\|G_j(y)\|
    =
    \frac{\kappa_\alpha}{\alpha}\|y-\widehat y\|,
\]
which proves \eqref{eq:fixed_selected_pg_error_bound}.
\end{proof}

This proposition controls the least-norm element of the full set-valued
mapping, and hence every possible projection selected by the algorithm. Consequently,
the particular selection specified in \eqref{eq:scalar_projection_rule}
requires no separate error-bound argument.

\begin{corollary}
\label{cor:finite_stepsize_pg_error_bound}
Suppose that the assumptions of
Proposition~\ref{prop:pg_error_bound} hold, and let
\(\mathcal A\subseteq
[\underline\alpha,\overline\alpha]\) be finite and nonempty.  Then
\[
    \operatorname{dist}(y,\mathcal S)
    \leq
    \kappa_G g_\alpha(y),
    \qquad
    \forall y\in\mathcal L_0,\quad \forall \alpha\in\mathcal A,
\]
where
\(
    \kappa_G:=\max_{\alpha\in\mathcal A}\kappa_\alpha<\infty.
\)
Consequently, for every \(\alpha\in\mathcal A\) and every selected
projection
\(\widehat y\in\Pi_{\mathcal C}
(y-\alpha\nabla\Phi(y))\),
\begin{equation}
\label{eq:selected_pg_error_bound}
    \operatorname{dist}(y,\mathcal S)
    \leq
    \frac{\kappa_G}{\alpha}
    \left\|
        (y-\widehat y)
    \right\|, \qquad \forall y\in\mathcal L_0.
\end{equation}
\end{corollary}

\begin{proof}
Apply Proposition~\ref{prop:pg_error_bound} to every
\(\alpha\in\mathcal A\) and take the maximum of the finitely many resulting
constants.
\end{proof}

\begin{remark}
\label{rem:pg_error_bound_scope}
The finite-set condition includes a fixed stepsize and finite-grid
backtracking rules while retaining variable stepsizes.  If the stepsizes
range over the whole interval
\([\underline\alpha,\overline\alpha]\), the same conclusions remain valid
provided that the Hoffman constants of the finitely many polyhedral pieces
are uniformly bounded on
that interval.  The finite-set condition avoids imposing this additional
uniform bound on the main convergence theory.
\end{remark}

With sufficient descent and the projected-gradient error bound, 
both objective decrease and distance to the solution set can be derived.

\begin{theorem}
\label{thm:error_bound_linear_convergence}
Suppose that Assumption~\ref{ass:stepsize} and the assumptions of
Proposition~\ref{prop:pg_error_bound} hold, and that
\(\alpha_k\in\mathcal A\) for all \(k\), where
\(\mathcal A\subseteq[\underline\alpha,\overline\alpha]\) is finite and
nonempty.
Let \(\kappa_r>0\) satisfy \eqref{eq:abstract_residual_error_bound}. Let \(\kappa_G\) be given by
Corollary~\ref{cor:finite_stepsize_pg_error_bound}, and enlarge
\(\kappa_G\), if necessary, so that
\(
    \eta
    :=
    \frac{2c_d\underline\alpha^2}{L\kappa_G^2}
    \in(0,1)
\)
and \(\theta:=1-\eta\in(0,1).\)
Then the safeguarded sequence satisfies
\begin{align}
    \Phi(y^k)
    &\leq
    \theta^k\Phi(y^0),
    \label{eq:linear_objective_rate}\\
    \operatorname{dist}(y^k,\mathcal S)
    &\leq
    \kappa_r\sqrt{2\Phi(y^0)}\,\theta^{k/2}.
    \label{eq:linear_iterate_rate}
\end{align}
\end{theorem}

\begin{proof}
Let
\[
    G_k
    :=
    G_{\alpha_k}(y^k;\bar y^{k+1})
    =
    \frac{1}{\alpha_k}(y^k-\bar y^{k+1}).
\]
Choose \(y_*^k\in\Pi_{\mathcal S}(y^k)\).  Since
\(r(y_*^k)=0\), \(\|H\|_2^2=L\), and
\eqref{eq:selected_pg_error_bound} holds,
\[ 
    \Phi(y^k)
    =
    \frac12\|H(y^k-y_*^k)\|^2
    \leq
    \frac{L}{2}\operatorname{dist}^2(y^k,\mathcal S)
    \leq
    \frac{L\kappa_G^2}{2}\|G_k\|^2.
\]
Lemma~\ref{lem:pg_descent} and
\(\alpha_k\geq\underline\alpha\) give
\[
    \Phi(y^{k+1})
    \leq
    \Phi(y^k)-c_d\underline\alpha^2\|G_k\|^2
    \leq
    \theta\Phi(y^k).
\]
Induction proves \eqref{eq:linear_objective_rate}.  The residual error bound
\eqref{eq:abstract_residual_error_bound} then gives
\eqref{eq:linear_iterate_rate}.
\end{proof}

The theorem demonstrates linear convergence of the accepted safeguarded iterates and geometric decay of the distance to
$\mathcal S$. Together with
\(\Phi(\bar y^{k+1})\leq\Phi(y^k)\) and
\eqref{eq:abstract_residual_error_bound}, it also gives
\(
    \operatorname{dist}(\bar y^{k+1},\mathcal S)
    \leq
    \kappa_r\sqrt{2\Phi(y^0)}\,\theta^{k/2}.
\)
Under part~(ii) of Proposition~\ref{prop:problem_residual_error_bounds}, one
may take \(\kappa_r=\mu_C^{-1}\). Corollary~
\ref{cor:gave_lcp_residual_error_bounds} gives the corresponding choices
\(\kappa_r=\mu_{\rm G}^{-1}\) for the square GAVE and
\(\kappa_r=\mu_{\rm L}^{-1}\) for the P-matrix LCP.

\subsection{Iteration bounds for active-face identification and finite termination}
\label{subsec:quantitative_active_face_identification}

Linear convergence remains asymptotic and does not imply that an iterate ever
belongs to $\mathcal S$. Finite termination requires two further steps: faces
disjoint from $\mathcal S$ must be excluded, and the refinement must then be
invoked on a remaining face. We quantify these steps using the deterministic
face $F(y)$ and the finite collection $\mathfrak F$ from
Section~\ref{sec:preliminaries}. If \(\mathcal S\neq\varnothing\) and
Assumption~\ref{ass:bounded_initial_level_set} holds, then
\(\mathcal S\) is a compact subset of \(\mathcal L_0\). In this case, define
\[
    \mathfrak F_{\mathcal S}
    :=
    \{F\in\mathfrak F:F\cap\mathcal S\neq\varnothing\}.
\]
If \(\mathfrak F_{\mathcal S}\) is a singleton, denote its unique element
by \(F_{\mathcal S}\), so that
\(\mathfrak F_{\mathcal S}=\{F_{\mathcal S}\}\).
If \(\mathfrak F\setminus\mathfrak F_{\mathcal S}\neq\varnothing\), set
\(
    \gamma_{\mathcal S}
    :=
    \min_{F\in\mathfrak F\setminus\mathfrak F_{\mathcal S}}
    \operatorname{dist}(F,\mathcal S).
\)
Since every face is closed, compactness of \(\mathcal S\) and finiteness of
\(\mathfrak F\) imply \(\gamma_{\mathcal S}>0\). If
\(\mathfrak F_{\mathcal S}=\mathfrak F\),
\(\gamma_{\mathcal S}=+\infty\) by convention.

The following elementary lemma converts the positive separation margin into
a face-identification criterion.

\begin{lemma}
\label{lem:active_face_separation}
Suppose that \(\mathcal S\neq\varnothing\) and
Assumption~\ref{ass:bounded_initial_level_set} holds. Let \(y\in\mathcal C\).
If
\(
    \operatorname{dist}(y,\mathcal S)<\gamma_{\mathcal S},
\)
then \(F(y)\in\mathfrak F_{\mathcal S}\).
\end{lemma}

\begin{proof}
If \(F(y)\notin\mathfrak F_{\mathcal S}\), then
\(y\in F(y)\) and
\[
    \operatorname{dist}(y,\mathcal S)
    \geq
    \operatorname{dist}(F(y),\mathcal S)
    \geq
    \gamma_{\mathcal S},
\]
which is a contradiction.
\end{proof}

The main point is now to combine this fixed geometric margin with the
iteration-dependent distance estimate. This gives one index after which no
face disjoint from $\mathcal S$ can contain an iterate or a projected-gradient
point.

\begin{theorem}
\label{thm:quantitative_active_face_identification}
Suppose that the assumptions of
Theorem~\ref{thm:error_bound_linear_convergence} hold, and set
\(C_0:=\kappa_r\sqrt{2\Phi(y^0)}\). If
\(\gamma_{\mathcal S}<+\infty\), let \(K_{\rm id}\) be the smallest
nonnegative integer satisfying
\(C_0\theta^{K_{\rm id}/2}<\gamma_{\mathcal S}\); otherwise, set
\(K_{\rm id}=0\). Then, for every
\(k\geq K_{\rm id}\),
\begin{equation}
\label{eq:iteration_bound}
    F(y^k),F(\bar y^{k+1})\in\mathfrak F_{\mathcal S}.
\end{equation}
If \(\gamma_{\mathcal S}<+\infty\), then
\begin{equation}
\label{eq:identification_iteration_bound}
    K_{\rm id}
    \leq
    1+
    \left\lceil
    \frac{
        2\log\left(
        \max\left\{1,2C_0/\gamma_{\mathcal S}\right\}
        \right)
    }{
        \log(1/\theta)
    }
    \right\rceil.
\end{equation}
If \(\mathfrak F_{\mathcal S}\) is a singleton, then
\begin{equation}
\label{eq:permanent_face_identification}
    F(y^k)
    =
    F(\bar y^{k+1})
    =
    F_{\mathcal S},
    \qquad
    \forall k\geq K_{\rm id}.
\end{equation}
Moreover, there exists
\(\vartheta_{\mathcal S}\in\{-1,1\}^n\) such that
\begin{equation}
\label{eq:permanent_activity_vector_identification}
    \vartheta(y^k)
    =
    \vartheta(\bar y^{k+1})
    =
    \vartheta_{\mathcal S},
    \qquad
    \forall k\geq K_{\rm id}.
\end{equation}
\end{theorem}

\begin{proof}
If \(\gamma_{\mathcal S}=+\infty\), then
\(\mathfrak F=\mathfrak F_{\mathcal S}\), and the result follows from
\(K_{\rm id}=0\).

Suppose that \(\gamma_{\mathcal S}<+\infty\). Since
$\Phi(\bar y^{k+1})\leq\Phi(y^k)$,
\eqref{eq:linear_objective_rate} and
\eqref{eq:abstract_residual_error_bound} imply
\[
    \max\left\{
    \operatorname{dist}(y^k,\mathcal S),
    \operatorname{dist}(\bar y^{k+1},\mathcal S)
    \right\}
    \leq C_0\theta^{k/2}.
\]
This inequality and the definition of
\(K_{\rm id}\) yield, for every \(k\geq K_{\rm id}\), 
\begin{equation}
\label{eq:post_identification_distance_envelope}
\max\left\{
    \operatorname{dist}(y^k,\mathcal S),
    \operatorname{dist}(\bar y^{k+1},\mathcal S)
\right\}
\leq
C_0\theta^{k/2}
\leq
C_0\theta^{K_{\rm id}/2}
<
\gamma_{\mathcal S}.
\end{equation}
Let \(v\in\{y^k,\bar y^{k+1}\}\).  By Lemma \ref{lem:active_face_separation}, 
\eqref{eq:iteration_bound} holds.
If \(\mathfrak F_{\mathcal S}\) is a singleton, the inclusion
reduces to \eqref{eq:permanent_face_identification}.
In this situation,
let \(v\in\{y^k,\bar y^{k+1}\}\) for some \(k\geq K_{\rm id}\).  Suppose
\(v\) is biactive at an index \(i\), write
\(F(v)=F_s=F_{\mathcal S}\) with an indicating vector $s$. Define another indicator \(s'\in\{-1,1\}^n\) by
\(s'_i=-s_i\) and \(s'_j=s_j\) for \(j\neq i\).  Then
\(F':=F_{s'}\) is distinct from \(F_{\mathcal S}\), while
\(v\in F'\) because \(u_i=v_i=0\). Since
\(\mathfrak F_{\mathcal S}=\{F_{\mathcal S}\}\), it follows that
\(F'\notin\mathfrak F_{\mathcal S}\), and hence
\[
    \operatorname{dist}(v,\mathcal S)
    \geq
    \operatorname{dist}(F',\mathcal S)
    \geq
    \gamma_{\mathcal S},
\]
contradicting \eqref{eq:post_identification_distance_envelope}.  Hence both
\(y^k\) and \(\bar y^{k+1}\) are strictly complementary.  On strictly
complementary points, \(\vartheta(v)=\chi(v)\), and
\(F(v)=F_{\mathcal S}\) uniquely determines this common vector. This proves
\eqref{eq:permanent_activity_vector_identification}.
\end{proof}

\begin{remark}
\label{rem:identification_margin}
The index \(K_{\rm id}\) bounds the number of projected-gradient iterations
required to exclude every face disjoint from \(\mathcal S\). This bound
deteriorates as \(\gamma_{\mathcal S}\downarrow0\), consistently with the
behavior near biactive points. If a solution is biactive at some coordinate,
it belongs to both incident faces, and hence
\(\mathfrak F_{\mathcal S}\) cannot be a singleton. Nevertheless,
\eqref{eq:iteration_bound} remains valid because it only excludes faces
disjoint from \(\mathcal S\). More precisely,
\(\mathfrak F_{\mathcal S}=\{F_{\mathcal S}\}\) holds if and only if
every solution is strictly complementary and all solutions induce the same
active face \(F_{\mathcal S}\).
\end{remark}

Once all faces disjoint from $\mathcal S$ have been excluded, either trigger
in \eqref{eq:refinement_trigger} can invoke the refinement on a face
intersecting $\mathcal S$. 
The following corollary considers this to show the finite termination.
 
\begin{corollary}
\label{cor:finite_termination_complexity}
Suppose that the assumptions of
Theorem~\ref{thm:quantitative_active_face_identification} hold. Assume that an
invoked refinement returns a solution whenever the selected face intersects
\(\mathcal S\), in the sense that
\begin{equation}
\label{eq:selected_face_refinement_solution}
    \mathcal T_k
    \ \text{and}\
    F(\bar y^{k+1})\in\mathfrak F_{\mathcal S}
    \quad\Longrightarrow\quad
    \widehat y^{k+1}\ \text{is well-defined and}\
    \widehat y^{k+1}\in\mathcal S,
    \qquad k\geq0.
\end{equation}
\begin{enumerate}
    \item[(i)] Suppose that the residual trigger is used with \(\tau>0\).
    Using the normalization \(s_c\) from
    \eqref{eq:relative_residual_scaling}, let \(K_\tau\) be the smallest
    nonnegative integer satisfying
    \begin{equation}
    \label{eq:residual_trigger_index}
        \frac{\sqrt{2\Phi(y^0)}}{s_c}
        \theta^{K_\tau/2}
        \leq\tau.
    \end{equation}
    Then Algorithm \ref{alg:main} reaches \(\mathcal S\) no later than
    \begin{equation}
    \label{eq:residual_finite_termination_bound}
        K^{(\rho)}
        :=
        \max\{K_{\rm id},K_\tau\}+1,
    \end{equation}
    where \(K_{\rm id}\) is defined in
    Theorem~\ref{thm:quantitative_active_face_identification}, and
    \[
        K_\tau
        \leq
        \left\lceil
        \frac{
            2\log\left(
            \max\left\{
            1,\sqrt{2\Phi(y^0)}/(s_c\tau)
            \right\}
            \right)
        }{
            \log(1/\theta)
        }
        \right\rceil.
    \]

    \item[(ii)] If
    \(\mathfrak F_{\mathcal S}\) is a singleton, then Algorithm \ref{alg:main} reaches
    \(\mathcal S\) no later than
    \begin{equation}
    \label{eq:stability_finite_termination_bound}
        K^{(\rm stab)}
        :=
        K_{\rm id}+q.
    \end{equation}
\end{enumerate}
Here \(q\geq1\) is the activity-vector stability threshold appearing in
\eqref{eq:refinement_trigger}.
\end{corollary}

\begin{proof}
Set \(k_\rho:=\max\{K_{\rm id},K_\tau\}\). By the definition of $\mathcal T_k$ in \eqref{eq:refinement_trigger}, along with
Lemma \ref{lem:pg_descent}, \eqref{eq:linear_objective_rate}, and
\eqref{eq:residual_trigger_index}, we have
\[
    \rho(\bar y^{k_\rho+1})
    =
    \frac{\sqrt{2\Phi(\bar y^{k_\rho+1})}}{s_c}
    \leq
    \frac{\sqrt{2\Phi(y^{k_\rho})}}{s_c}
    \leq
    \frac{\sqrt{2\Phi(y^0)}}{s_c}\theta^{k_\rho/2}
    \leq
    \frac{\sqrt{2\Phi(y^0)}}{s_c}\theta^{K_\tau/2}
    \leq\tau.
\]
Applying Theorem~\ref{thm:quantitative_active_face_identification} yields
\(F(\bar y^{k_\rho+1})\in\mathfrak F_{\mathcal S}\).
Since \(\mathcal T_{k_\rho}\) holds,
\eqref{eq:selected_face_refinement_solution} gives
\(\widehat y^{k_\rho+1}\in\mathcal S\).  Hence
\(\Phi(\widehat y^{k_\rho+1})=0
\leq\Phi(\bar y^{k_\rho+1})\), and 
\eqref{eq:refinement_safeguard} gives
\(y^{k_\rho+1}=\widehat y^{k_\rho+1}\in\mathcal S\).
Thus Algorithm \ref{alg:main} reaches \(\mathcal S\) by iteration
\(k_\rho+1=K^{(\rho)}\).

For part~(ii), let
\(\mathfrak F_{\mathcal S}=\{F_{\mathcal S}\}\).
Theorem~\ref{thm:quantitative_active_face_identification} shows that
\(y^k\) and \(\bar y^{k+1}\) are strictly complementary for every
\(k\geq K_{\rm id}\). Consequently,
\(\vartheta(v)=\chi(v)\) for
\(v\in\{y^k,\bar y^{k+1}\}\). For every
\(j=0,\ldots,q-1\),
\eqref{eq:permanent_activity_vector_identification} gives
\begin{equation}
\label{eq:post_identification_activity_match}
    \bar\vartheta_{K_{\rm id}+j+1}
    =\vartheta_{K_{\rm id}+j},
    \qquad j=0,\ldots,q-1.
\end{equation}
This means that the consecutive iterations are in the same face. By the definition of $\mathcal T_k$,
\(\nu_{K_{\rm id}+j+1}
=\nu_{K_{\rm id}+j}+1\) for \(j=0,\ldots,q-1\).
Since \(\nu_{K_{\rm id}}\geq0\), we have
\(\nu_{K_{\rm id}+q}\geq q\).
Thus \(\mathcal T_{K_{\rm id}+q-1}\) holds, and the refinement generated
at iteration \(k=K_{\rm id}+q-1\) is invoked. Moreover,
\eqref{eq:iteration_bound} gives
\(F(\bar y^{K_{\rm id}+q})\in\mathfrak F_{\mathcal S}\).  Therefore,
\(\widehat y^{K_{\rm id}+q}\in\mathcal S\) as we assume the refinement returns a solution.  Hence
\(\widehat y^{K_{\rm id}+q}\in\mathcal C\) and
\(\Phi(\widehat y^{K_{\rm id}+q})=0
\leq\Phi(\bar y^{K_{\rm id}+q})\).  The safeguard
\eqref{eq:refinement_safeguard} therefore gives
\(y^{K_{\rm id}+q}=\widehat y^{K_{\rm id}+q}\in\mathcal S\).
This proves
\(K^{(\rm stab)}=K_{\rm id}+q\).
\end{proof}

\begin{remark}
\label{rem:finite_termination_context}
\begin{itemize}
    \item[(i)]
    To verify condition \eqref{eq:selected_face_refinement_solution} in
    Corollary~\ref{cor:finite_termination_complexity}, consider any
    \(F_s\in\mathfrak F_{\mathcal S}\).  Since
    \(F_s\cap\mathcal S\neq\varnothing\), the system
    \eqref{eq:selected_face_refinement_system} admits a solution
    \(\widetilde\xi\) satisfying \(D_s\widetilde\xi\geq0\).  If \(m>n\), full
    column rank of \(C_{\Theta_s}\) makes this solution unique; for a GAVE, this
    follows from \(\mu_C>0\).  If \(m=n\), the corresponding sufficient condition
    is nonsingularity, which follows from \(\mu_{\rm G}>0\) for a GAVE and from
    \(\mu_{\rm L}>0\) for a P-matrix LCP.  If \(m<n\), the
    refinement must explicitly select a solution satisfying
    \(D_s\widetilde\xi\geq0\).  In all cases, the selected solution is converted
    into the corresponding point of \(F_s\) as in
    \eqref{eq:selected_face_refinement_point}.  
    For part~(ii) of Corollary~\ref{cor:finite_termination_complexity}, 
    the condition that \(\mathfrak F_{\mathcal S}\) 
    being a singleton is equivalent to every solution
    being strictly complementary and all solutions inducing the same face.
    Part~(i), in contrast, does not require strict complementarity.

    \item[(ii)]
    The closest related result is the component-identification principle of
    \cite{alcantara2025global}, but it does not provide an explicit
    identification index. Our result provides an explicit
    face-identification bound and, when combined with exact selected-face
    refinement, explicit finite-termination bounds.
    Newton-type finite-termination results rely on different, algorithm-specific
    conditions.  
    Fischer and Kanzow obtain local one-step termination from
    nonsingularity of the relevant generalized Jacobians under conditions such
    as P-matrix, R-regularity, or b-regularity, while allowing degenerate
    indices~\cite{fischer1996finite}. Their one-step result is local, and does not
    provide the number of iterations required to enter that neighborhood.
    Sun and Huang establish finite termination
    at a maximally complementary solution for sufficient-matrix LCPs
    \cite{sun2006smoothing}, but provide no explicit bound on the 
    termination index.
    The conjugate-gradient method of \cite{li2008conjugate} has the explicit bound
    \(
        \lvert J(x^*)\rvert\leq n
    \)
    on the number of conjugate-gradient steps, but it requires
    the symmetric M-matrix structure, which is a subset of our P-matrix.
\end{itemize}
\end{remark}

\section{Numerical Experiments}\label{sec:numerics}

In this section, we evaluate the proposed method, denoted by PGDN, on GAVEs and LCPs. Since
the two model classes use different variables and residuals, comparisons are
made within each class.

\subsection{Implementation details}

The global convergence framework uses the sufficient condition
$\alpha_k<1/L$, and the linear-rate results additionally assume that the
stepsizes belong to a finite set. Our projected-gradient step uses alternating Barzilai--Borwein trial
stepsizes \cite{barzilai1988two}, globalized by a projected Armijo line
search. The trial value is truncated 
to $[\alpha_{\min},\alpha_{\max}]$ and then reduced by the factor $\beta$
until the projected-gradient trial point satisfies the Armijo condition
\[
    \Phi(\bar y^{k+1})
    \leq
    \Phi(y^k)
    -\frac{\sigma}{\alpha_k}
    \|\bar y^{k+1}-y^k\|^2.
\]
Thus, the implementation does not compute the Lipschitz constant. However, as the objective function $\Phi(y)$ is $L-$smooth, the BB stepsize and the following stepsize produced by line search automatically share the theoretical upper bound $1/L$. Of course, other choices of stepsize are also possible. The default values are
$\alpha_0=1$, $\alpha_{\min}=10^{-12}$,
$\alpha_{\max}=10^{12}$, $\beta=0.5$, and $\sigma=10^{-4}$, unless stated
otherwise.

The refinement is triggered when either the normalized equation residual
is at most \texttt{newtonTol} or the vector $\vartheta(y)$ has been
unchanged for \texttt{signStableIts} consecutive iterations. The defaults
are \texttt{newtonTol}$=10^{-3}$, \texttt{signStableIts}$=5$, and at most
\texttt{maxNewton}$=3$ inner refinements. Damped candidates are tested in
the order $\gamma\in\{1,0.5,0.25,0.1\}$ and retained only if the normalized
residual is at most \texttt{newtonEta}$=0.8$ times the best current
residual. The final refined point is used only when it improves the
projected-gradient trial point, which realizes the safeguard
\eqref{eq:refinement_safeguard}.

Both implementations specialize \eqref{eq:selected_face_refinement_system}. For a
GAVE iterate $x$, define
\[
    d_i(x):=
    \begin{cases}
        1,&x_i\geq0,\\
       -1,&x_i<0,
    \end{cases}
    \qquad
    D(x):=\operatorname{diag}(d(x)).
\]
The diagonal selector gives the system
\(\bigl(A-BD(x)\bigr)\widetilde x=b\).
A direct linear solver is used for square systems and a minimum-norm
least-squares solution for rectangular systems. Each damped point has the
form $x+\gamma(\widetilde x-x)$ and is tested using the true GAVE residual.
After acceptance, the complementarity variables are recovered from its
positive and negative parts.

For the direct LCP solver, let
$\mathcal I_w=\{i:w_i>0\}$ and
$\mathcal I_z=\{1,\ldots,n\}\setminus\mathcal I_w$. Hence, if
$z_i=w_i=0$, then $i\in\mathcal I_z$. Accordingly, the refinement fixes
$z_{\mathcal I_w}=0$ and $w_{\mathcal I_z}=0$ and solves
\(
    \begin{bmatrix}-M_{:,\mathcal I_z}&I_{:,\mathcal I_w}\end{bmatrix}
    \begin{bmatrix}
        \widetilde z_{\mathcal I_z}\\
        \widetilde w_{\mathcal I_w}
    \end{bmatrix}
    =-b.
\)
The resulting system is square. A candidate is discarded if an active
component is above the prescribed feasibility tolerance; the remaining
damped candidates are tested using
$\|w-Mz+b\|/\max\{1,\|b\|\}$.

Both solvers stop when the normalized equation residual reaches
\texttt{tolRes}. They also stop when the infinity norm of the
projected-gradient displacement is at most \texttt{tolStep} and no
refinement was accepted; in this case the termination flag records whether
the residual tolerance was met. Consequently, the practical solver may stop
at a prescribed tolerance before the selected-face linear system satisfies
the exact-solution condition
\eqref{eq:selected_face_refinement_solution}.

\subsection{Experimental settings}

All methods were implemented in MATLAB R2024b and run on a 64-bit workstation
with an Intel Core i9-12900 CPU at 2.40 GHz and 64 GB of RAM. CPU times were
measured in this common environment. Because the methods have different
per-iteration costs, iteration counts are compared only within the same
model class.

For GAVE instances, we report the relative equation residual and, when a
reference solution is available, the relative solution error:
\[
 \mathrm{relRes}:=\frac{\|Ax-B|x|-b\|_2}{\max\{1,\|b\|_2\}}, \qquad
 \mathrm{distX}:=\frac{\|x-x^\ast\|_2}{\max\{1,\|x^\ast\|_2\}}.
\]
For direct LCP instances, we report CPU time, the natural residual, and the
feasibility-and-complementarity residual defined below. PGDN is shown in red
in all figures.

\subsection{Ablation study}

We isolate the effects of the BB stepsize and algebraic refinement by
comparing four variants:
\begin{itemize}
    \item \textbf{PGD.} Projected gradient descent with a fixed stepsize and without algebraic refinement.
    \item \textbf{PGD-BB.} Projected gradient descent equipped with the Barzilai--Borwein (BB) stepsize.
    \item \textbf{PGD-Newton.} Projected gradient descent with the safeguarded algebraic refinement and a fixed stepsize.
    \item \textbf{PGDN.} The full proposed method.
\end{itemize}

All variants use the same projected line search and differ only in the two
components indicated above. For the fixed-step variants, we use
\(
    \alpha_{\rm fix}
    :=\frac{0.99}{2(\sigma_{\max}(A)+\|B\|_2)^2}.
\)
We set the iteration limit to $4,000$ and use
\(
\texttt{tolRes}=10^{-12}\;{\rm and}\;\texttt{tolStep}=10^{-10}.
\) 

We generate synthetic GAVEs $Ax-B|x|=b$. The matrices $A$ and $B$ have
linearly spaced prescribed singular values; their singular-vector factors
are obtained by QR factorizations of Gaussian matrices. The baseline
parameters are
\begin{equation}\label{eq:sem_param}
\sigma_{\min}(A)=2,\qquad \kappa(A)=2,\qquad \|B\|_2=1,\qquad \kappa(B)=10.
\end{equation}
For each instance, we sample $x^\ast$ and set
\(
b = Ax^\ast - B|x^\ast|,
\) 
so that the instance is consistent. All variants use the same random initial
point in each trial. We test
$m\in\{512,1024,2048,4096\}$, and $n\in\{64,128,256,512\}$,
with $n\leq m$. For each pair $(m,n)$, ten independent trials are performed.

\begin{figure}[htbp]
\centering
\includegraphics[width= \textwidth]{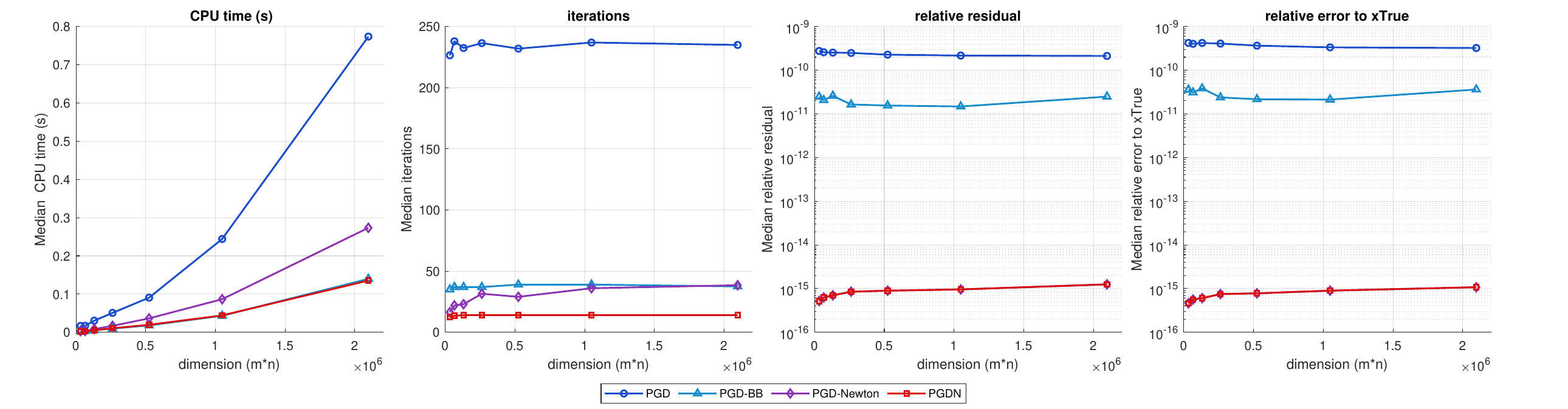}
\vspace{-.5cm}
\caption{Ablation results for the fixed-stepsize, BB, and safeguarded
refinement components. Each curve reports the median over independent trials
for the tested matrix dimensions.}
\label{fig:ablation}
\end{figure}

Figure~\ref{fig:ablation} shows the expected tradeoff. The theoretically safe
fixed stepsize makes PGD stable but substantially slower than its BB variant.
The refinement reduces both fixed-step and BB iterations and drives the final
residual and solution error close to machine precision. PGDN uses a modest
amount of additional linear algebra relative to PGD-BB, but it requires the
fewest outer iterations overall and is faster than both fixed-step variants.
Thus the BB rule accelerates the projected-gradient iterations, while the algebraic refinement
provides high final accuracy.

\subsection{Sensitivity analysis}

We next vary the condition numbers, the relative strength of the absolute-value
term, and the dimension. We set $\texttt{tolRes}=10^{-12}$,
$\texttt{tolStep}=10^{-10}$, and $x^0=0$, and report medians over five
independent trials.

For square instances with $n\in\{128,256,512\}$, we vary one parameter at a
time from the baseline \eqref{eq:sem_param} $\kappa(A),\kappa(B)\in\{2,5,10,20,50,100\}$ 
and $\|B\|_2/\sigma_{\min}(A)\in\{0.1,0.25,0.5,0.75,1.0\}.$
Figures~\ref{fig:sens_cpu_factor} and \ref{fig:sens_res_factor} report median
CPU time and relative residual, respectively.

\begin{figure}[htbp]
\centering
\includegraphics[width= .85\textwidth]{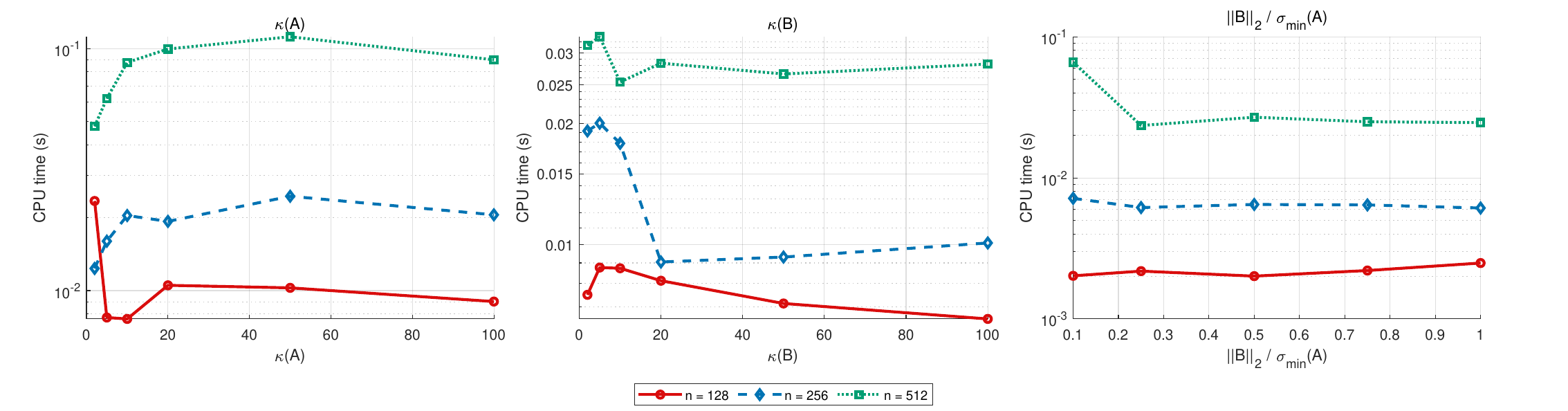}
\caption{Median CPU time of PGDN under variations in the condition numbers,
the ratio $\|B\|_2/\sigma_{\min}(A)$, and the problem dimension.}
\label{fig:sens_cpu_factor}
\end{figure}

\begin{figure}[htbp]
\centering
\includegraphics[width= .85\textwidth]{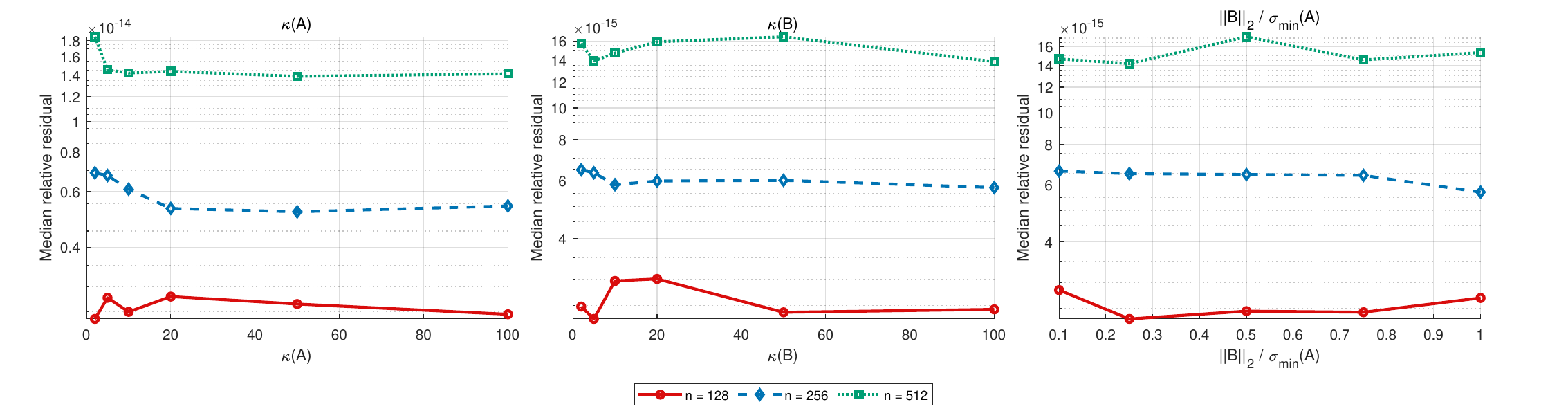}
\caption{Median relative residual of PGDN under variations in the condition
numbers, the ratio $\|B\|_2/\sigma_{\min}(A)$, and the problem dimension.}
\label{fig:sens_res_factor}
\end{figure}

Across the tested ranges, Figure~\ref{fig:sens_cpu_factor} shows moderate
variation in median CPU time, while Figure~\ref{fig:sens_res_factor} shows
that the final residual remains at a comparable level. The tested instances
therefore indicate limited sensitivity to these three factors.
 
\subsection{Comparisons with existing GAVE methods}

We compare PGDN with RABK~\cite{xie2025randomized},
PIM~\cite{rohn2014iterative}, MAP~\cite{alcantara2023method},
GNM~\cite{mangasarian2009generalized,hu2011generalized}, and
TAVE~\cite{abdallah2018solving}. Unless otherwise stated, all methods start
from $x^0=0$, use a maximum of 4000 iterations, and target a relative residual
of $10^{-12}$.

\paragraph{Square controlled-spectrum systems.}
Following \cite[Section~6.2]{xie2025randomized}, we set
$\sigma_{\min}(A)=2$ and $\|B\|_2=1$ and use
\[
 (\kappa(A),\kappa(B))
 \in\{(2,1),(3,10),(10,10),(30,100)\}.
\]
We test $m=n\in\{64,128,256,512,1024\}$, with three trials per size and
spectral pair.

\begin{figure}[p]
\centering
\begin{subfigure}{.49\textwidth}
\includegraphics[width=\textwidth]{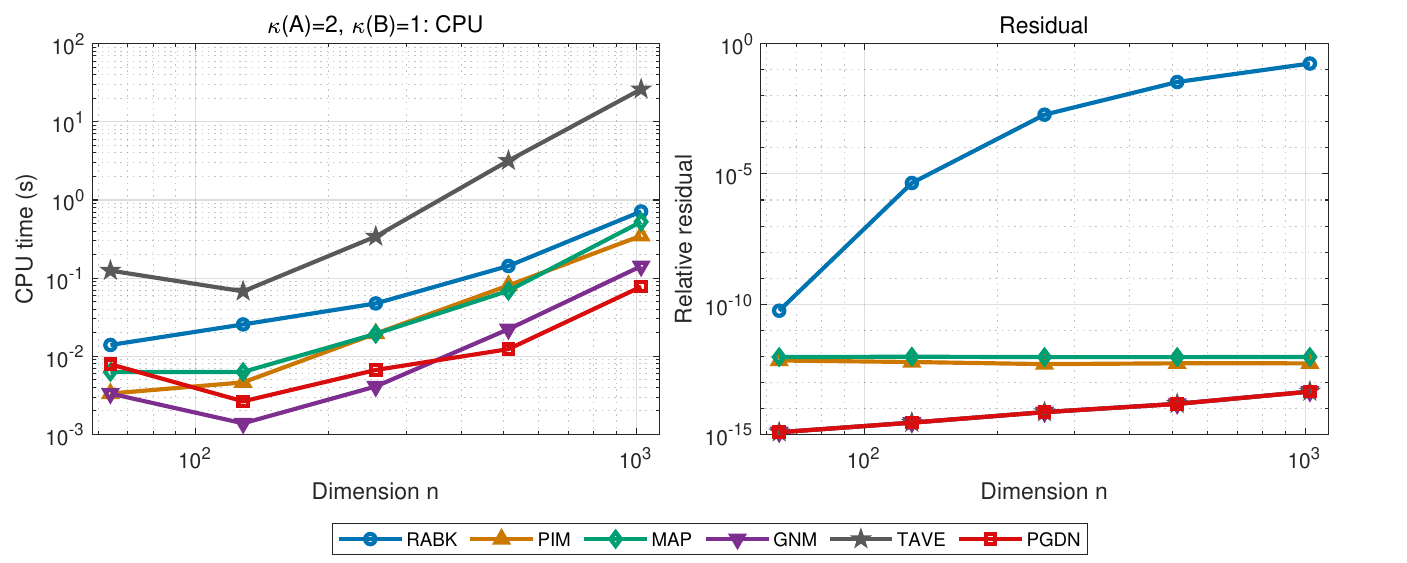}
\caption{$\kappa(A)=2$, $\kappa(B)=1$.}
\end{subfigure}
\begin{subfigure}{.49\textwidth}
\includegraphics[width=\textwidth]{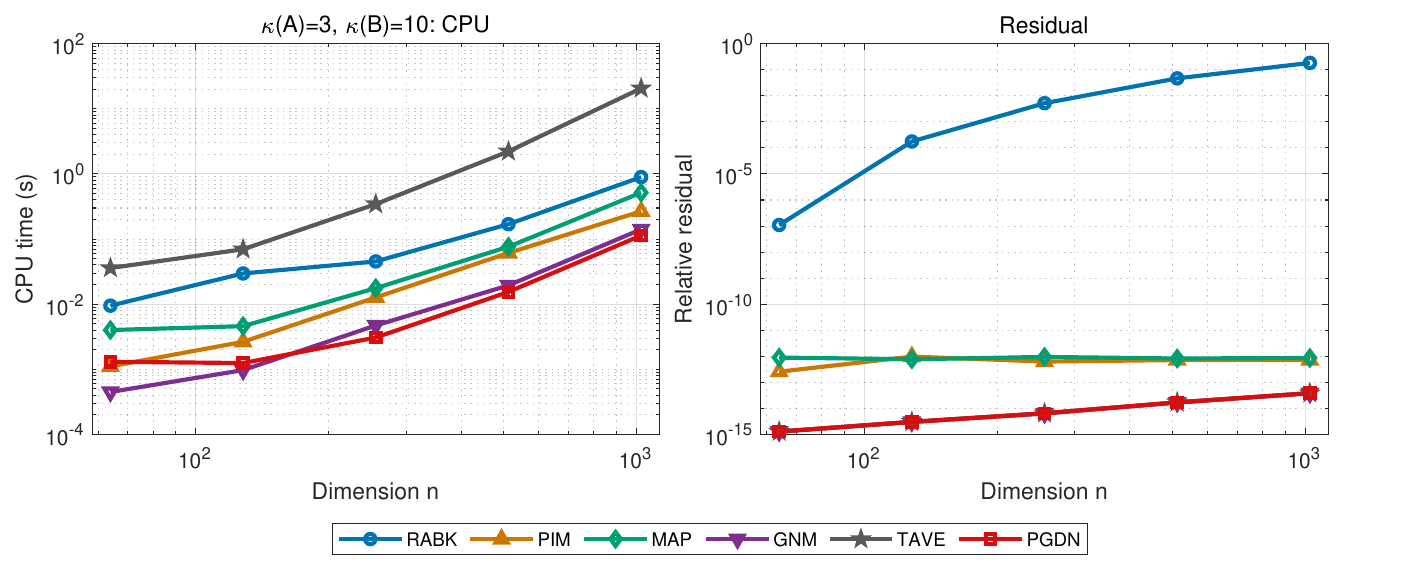}
\caption{$\kappa(A)=3$, $\kappa(B)=10$.}
\end{subfigure}
\begin{subfigure}{.49\textwidth}
\includegraphics[width=\textwidth]{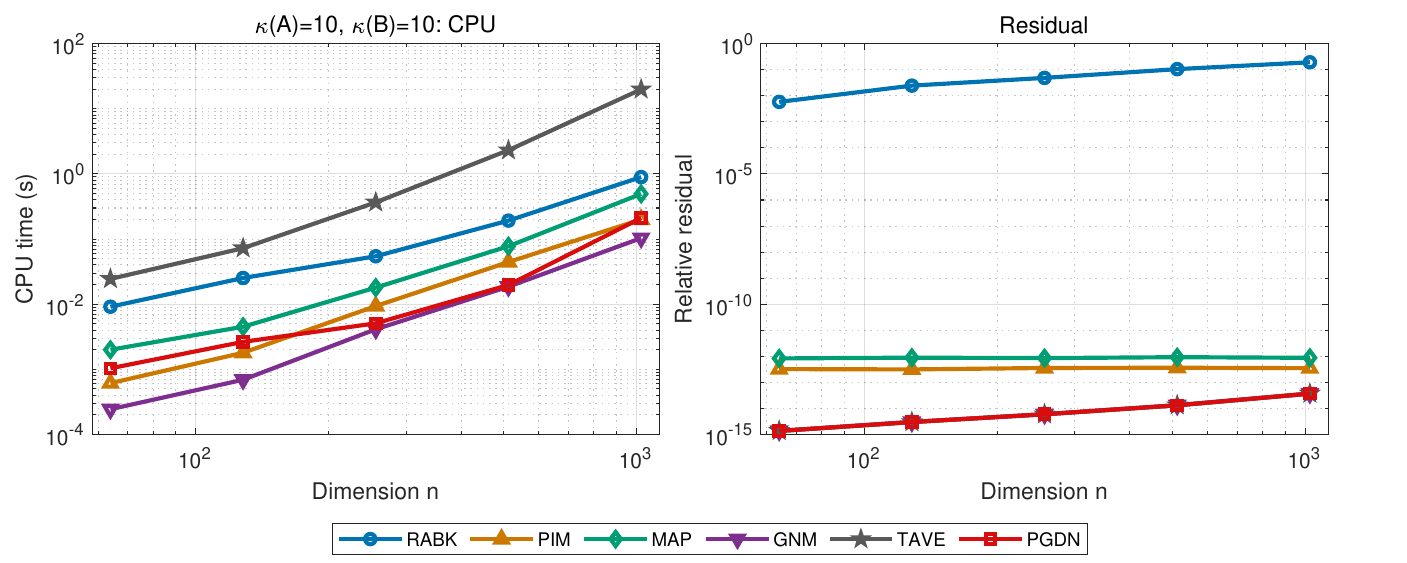}
\caption{$\kappa(A)=10$, $\kappa(B)=10$.}
\end{subfigure}
\begin{subfigure}{.49\textwidth}
\includegraphics[width=\textwidth]{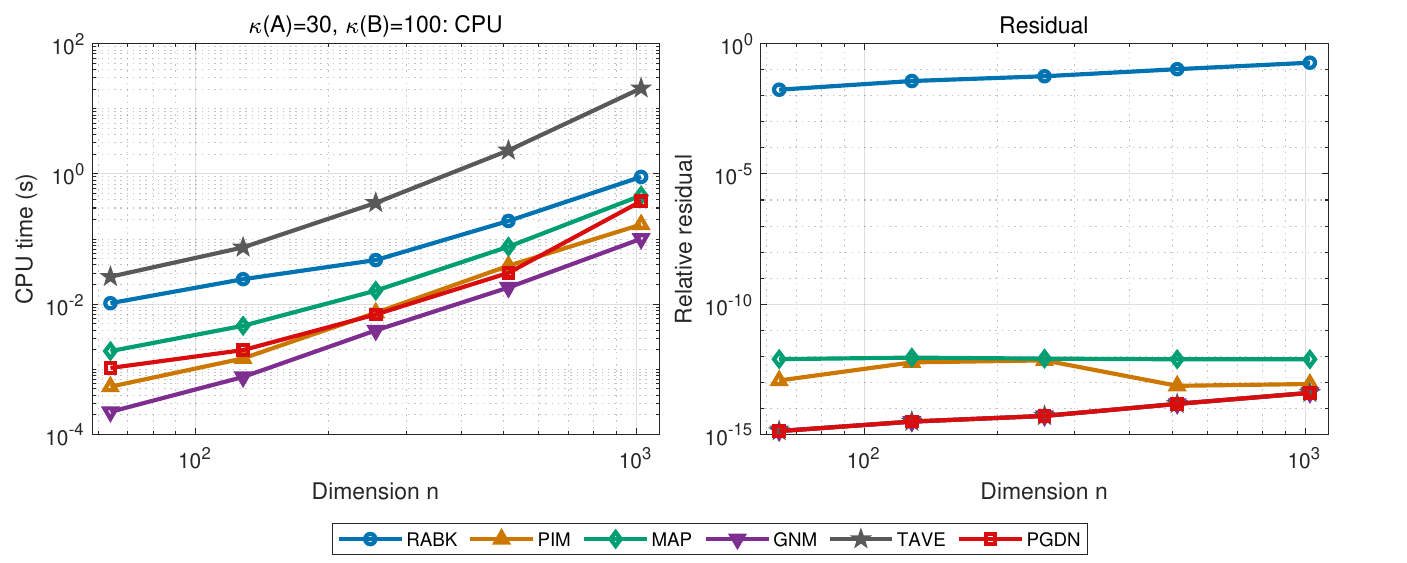}
\caption{$\kappa(A)=30$, $\kappa(B)=100$.}
\end{subfigure}
\caption{Square controlled-spectrum GAVE comparisons.}
\label{fig:compare-square}
\end{figure}

GNM and PGDN consistently attain residuals near floating-point precision.
PIM and MAP generally stop near the requested tolerance. RABK often reaches
4000 iterations before meeting the target, and its residual deteriorates as
dimension or conditioning increases. TAVE is accurate but becomes much more
expensive at $n=1024$. GNM or PIM is faster than PGDN in several settings, so
the results support high accuracy and competitive performance in selected
settings, but not uniform superiority in this case.

\paragraph{Tall controlled-spectrum systems.}
The rectangular experiment fixes $n=64$ and uses
\(
 m\in\{64,128,256,512,1024,2048\},
\)
with five trials per spectral pair. The comparison contains
RABK, MAP, and PGDN because their available implementations directly support
tall matrices. GNM requires a square system. PIM is defined through an
inverse solve with $A$; using MATLAB's rectangular backslash would replace
that step by a least-squares solve and would no longer implement the stated
PIM iteration.

The available SLA implementation~\cite{mangasarian2007absolute} constructs
and solves a dense LP at every iteration and has no enforceable time limit; it
already required more than one minute at $(m,n)=(128,64)$. The available TAVE
implementation likewise performs an LP initialization before its internal
time check. SLA and TAVE are therefore excluded rather than being assigned
misleading nominal timeouts.

\begin{figure}[p]
\centering
\begin{subfigure}{.49\textwidth}
\includegraphics[width=\textwidth]{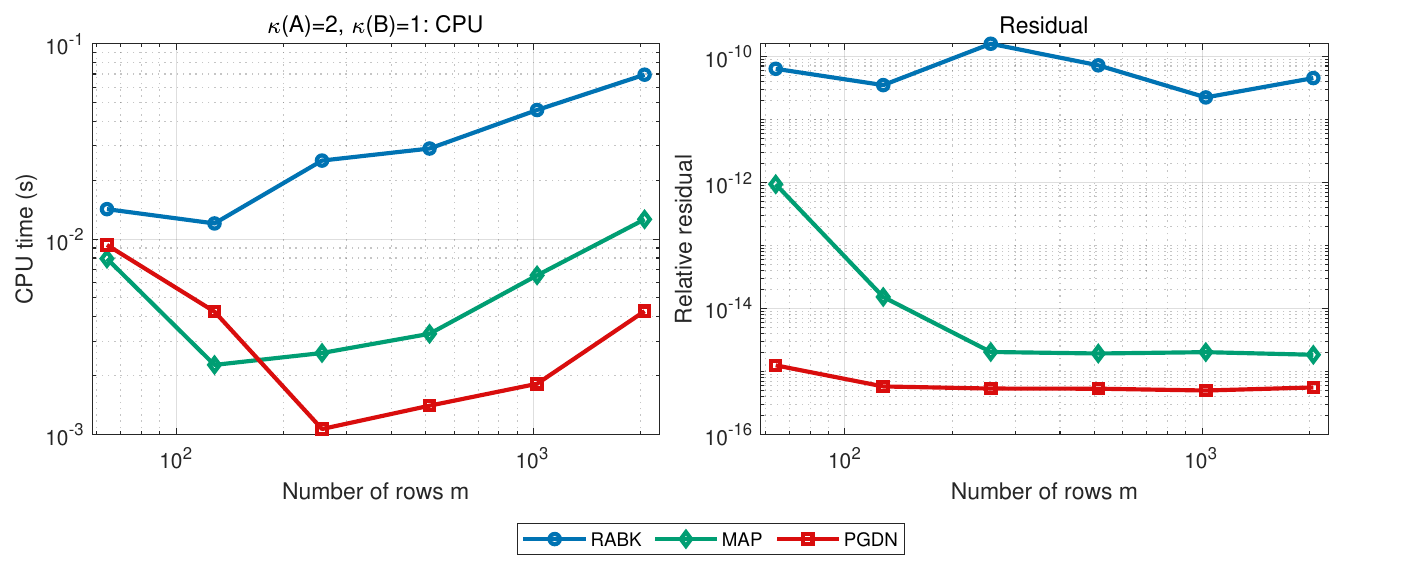}
\caption{$\kappa(A)=2$, $\kappa(B)=1$.}
\end{subfigure}
\begin{subfigure}{.49\textwidth}
\includegraphics[width=\textwidth]{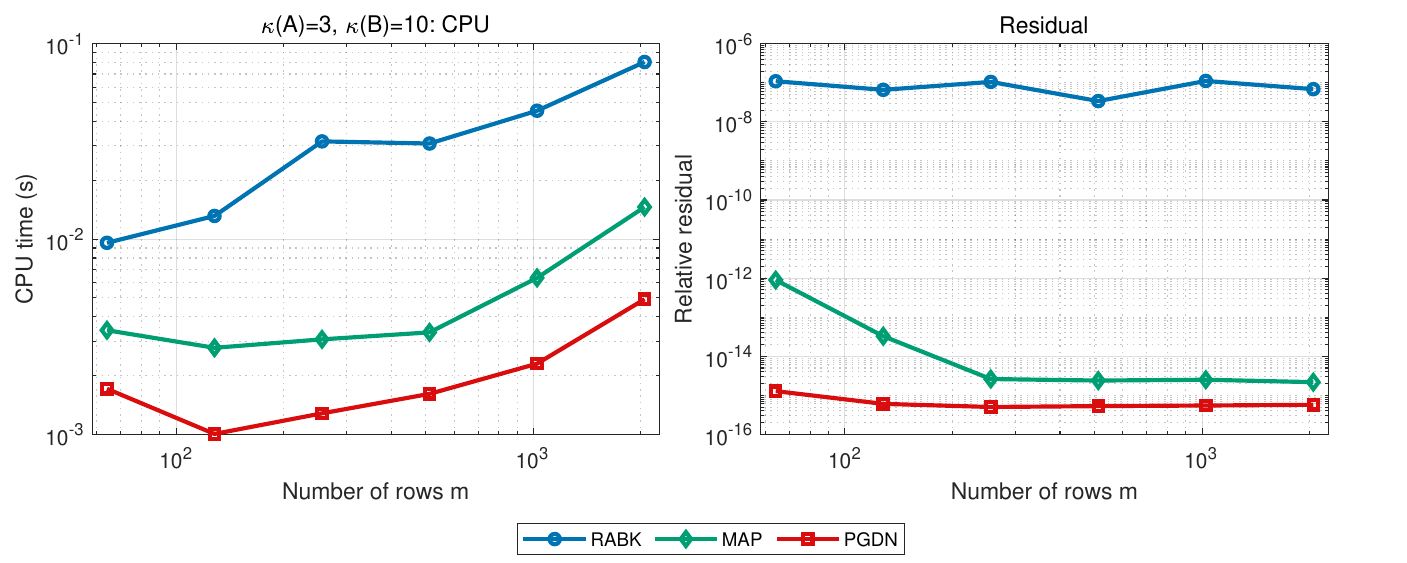}
\caption{$\kappa(A)=3$, $\kappa(B)=10$.}
\end{subfigure}
\begin{subfigure}{.49\textwidth}
\includegraphics[width=\textwidth]{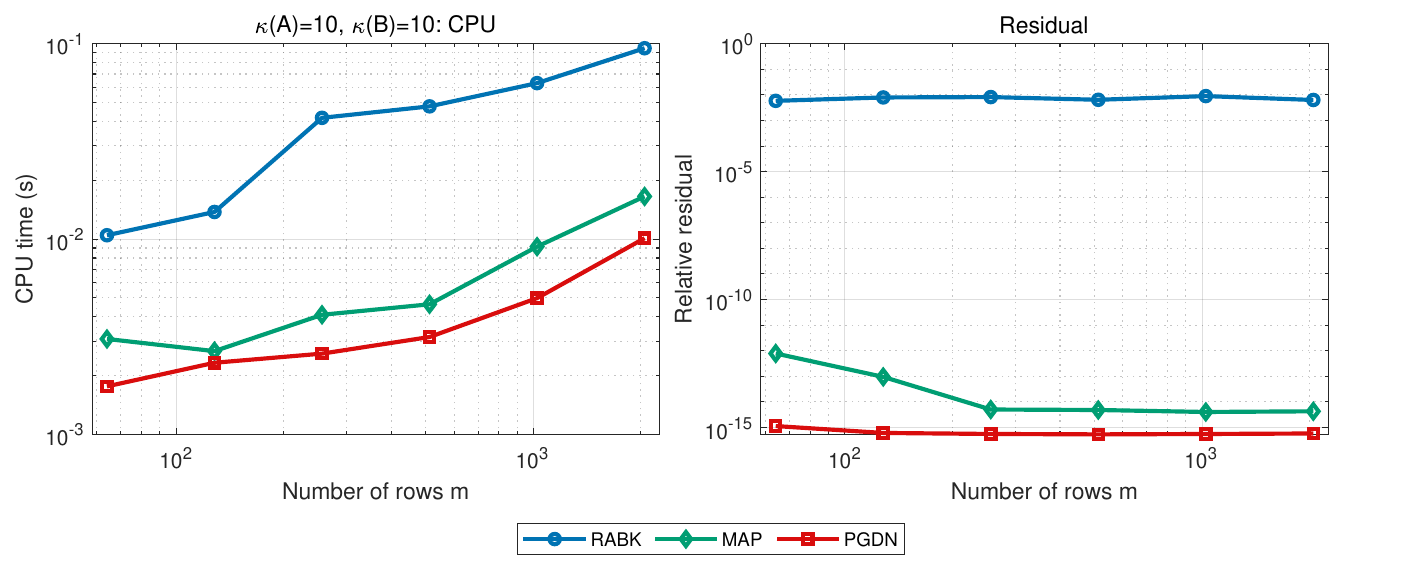}
\caption{$\kappa(A)=10$, $\kappa(B)=10$.}
\end{subfigure}
\begin{subfigure}{.49\textwidth}
\includegraphics[width=\textwidth]{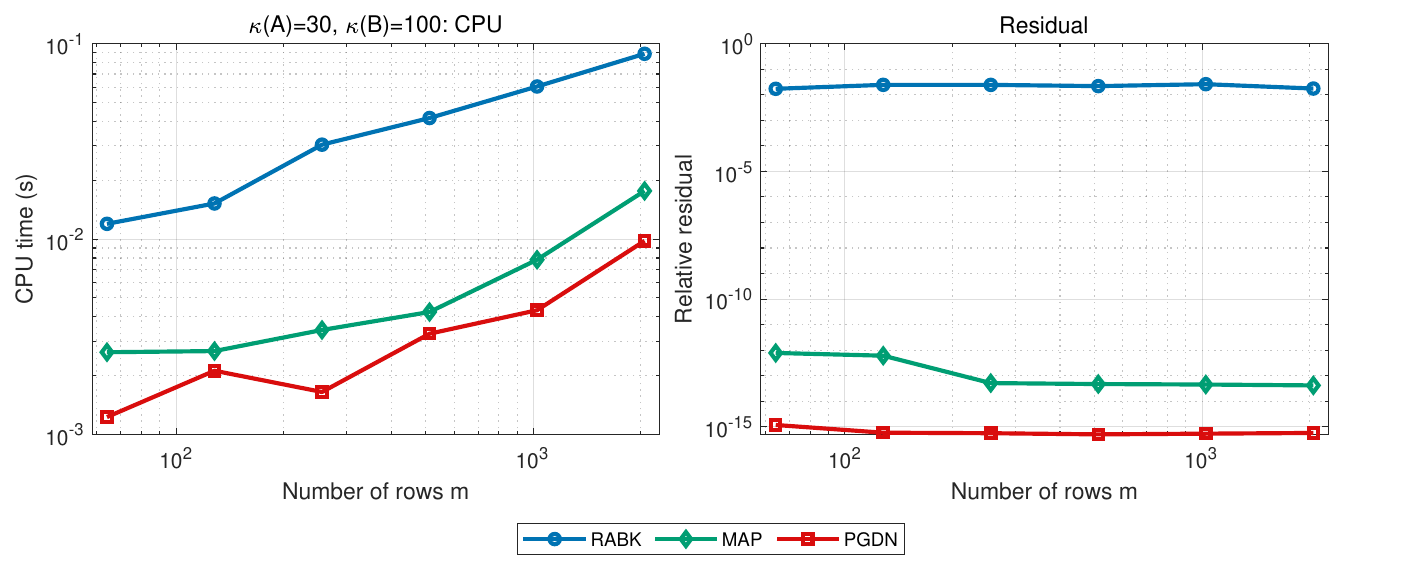}
\caption{$\kappa(A)=30$, $\kappa(B)=100$.}
\end{subfigure}
\caption{Tall controlled-spectrum GAVE comparisons with fixed $n=64$.}
\label{fig:compare-nonsquare}
\end{figure}

PGDN attains median residuals between approximately $4\times10^{-16}$ and
$10^{-15}$. MAP is also accurate, typically between $10^{-15}$ and
$10^{-13}$. RABK reaches its iteration limit and becomes less accurate for
larger or more ill-conditioned instances. At
$(\kappa(A),\kappa(B))=(30,100)$ and $(m,n)=(2048,64)$, the median CPU times
are 0.0098, 0.0176, and 0.0888 seconds for PGDN, MAP, and RABK, respectively;
their median residuals are $5.70\times10^{-16}$, $4.15\times10^{-14}$, and
$1.76\times10^{-2}$.

\paragraph{Large-scale sparse AVE systems.}
Following the sparse generation strategy in
\cite[Section~3]{bello2016global}, we use MATLAB's \texttt{sprand} with
prescribed singular values to generate $A\in\mathbb R^{n\times n}$ with requested
density $0.003$. We sample $x^\star$ and $x^0$ componentwise from
$U(-100,100)$ and set
\(
 b=Ax^\star-|x^\star|.
\)
The singular values are sampled from $[4,160]$, with both endpoints inserted,
so $\sigma_{\min}(A)=4$, $\kappa_2(A)=40$, and
$\|A^{-1}\|_2=1/4<1/3$. We use
$n\in\{1000,2500,5000,10000\}$ and three fixed-seed trials per dimension.

The comparison includes RABK~\cite{xie2025randomized},
PIM~\cite{rohn2014iterative},
Exact-SSN~\cite{mangasarian2009generalized}, Inexact-SSN~\cite{bello2016global},
and PGDN. 
Exact-SSN is the GNM iteration specialized to $B=I$. PIM reuses a sparse LU
factorization, Exact-SSN uses sparse direct solvers, and Inexact-SSN uses
warm-started LSQR with the residual-relative tolerance in
\cite{bello2016global}. The success criterion is
$\|Ax-|x|-b\|_2\le10^{-8}$. MAP and TAVE are omitted because their available
implementations do not preserve the intended large-scale sparse computational
model.

\begin{figure}[htbp]
  \centering
  \includegraphics[width=.7\textwidth]{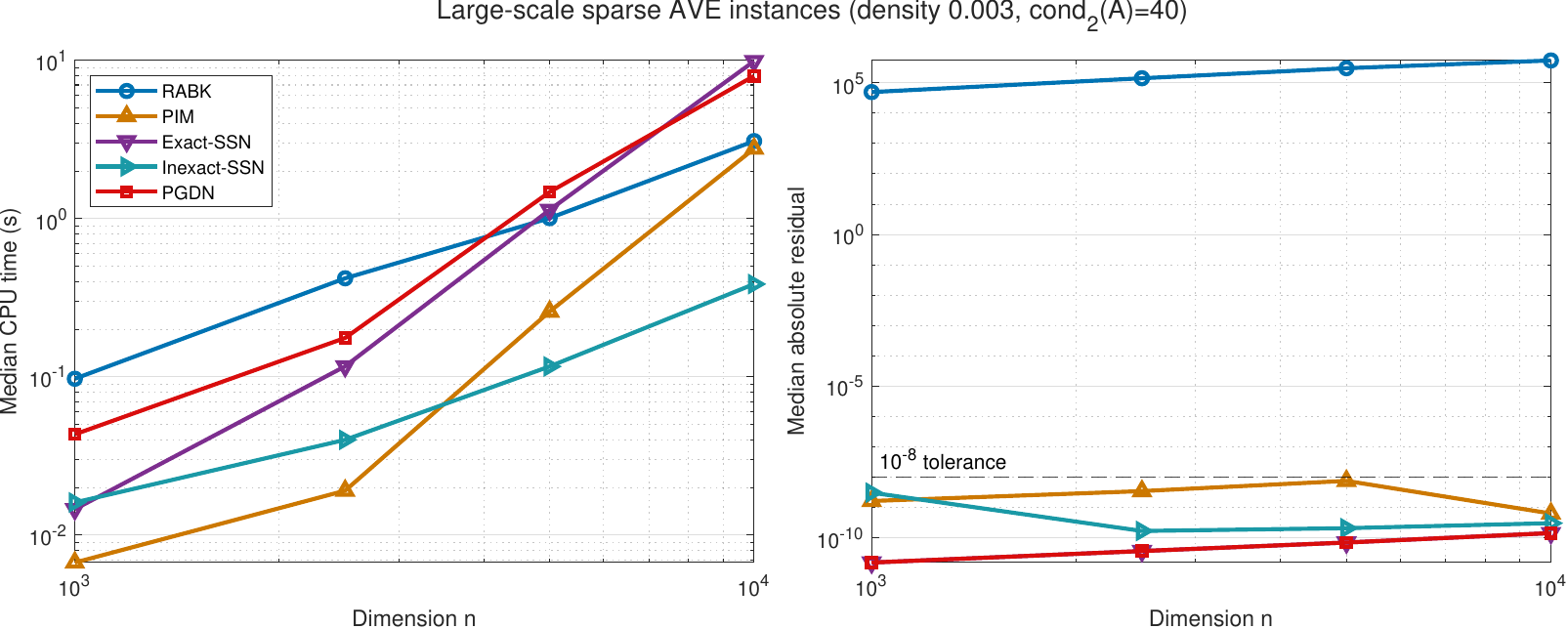}
  \caption{Median CPU time and absolute residual on large-scale sparse AVEs.}
  \label{fig:comparison-large-sparse-ave}
\end{figure}

PIM, Exact-SSN, Inexact-SSN, and PGDN solve every instance. RABK reaches its
4000-update limit on every instance and does not meet the tolerance. At
$n=10000$, Inexact-SSN is fastest, followed by PIM; PGDN and Exact-SSN are
more expensive because their accepted steps involve sparse direct
factorizations. PGDN nevertheless retains a median residual near
$10^{-10}$ in the absolute scale and near floating-point precision in the
relative scale. These results support accuracy and robustness for PGDN on
this controlled sparse family, but not superiority over specialized sparse
iterations.

\subsection{Comparisons with existing LCP methods}

We next solve three LCP benchmarks directly in the form
\[
  z\geq 0,\qquad w=Mz-b\geq 0,\qquad z^{\top}w=0.
\]
These tests assess the projection and
active-set refinement without converting the LCP to a GAVE.

Following \cite[Section~7.2]{alcantara2025global}, we use the following three LCP problems:
\begin{itemize}
\item \textbf{LCP1} \cite[Example 2]{qi2000newlook}: $M$ is tridiagonal with $M_{ii}=4$ for all $i$ and $M_{ij}=-1$ when $|i-j|=1$, while the benchmark right-hand side is the all-ones vector.
\item \textbf{LCP2} \cite[Example 7.1]{kanzow1996some}: $M$ is upper triangular with $M_{ii}=1$ for all $i$ and $M_{ij}=2$ for all $i<j$, and the benchmark right-hand side is again the all-ones vector.
\item \textbf{LCP3} \cite[Example 7.3]{kanzow1996some}: the benchmark right-hand side is sampled componentwise from the uniform distribution on $(-500,500)$, and
\(
  M=A_1^{\top}A_1+A_2+\operatorname{diag}(\eta),
\)
where the entries of $A_1$ and $A_2$ are independently sampled from the uniform distribution on $(-5,5)$, $A_2$ is skew-symmetric, and each entry of $\eta$ is independently sampled from $(0,0.3)$.
\end{itemize}

\begin{figure}[htbp]
  \centering
  \begin{subfigure}{\textwidth}
    \centering
    \includegraphics[width=.85\textwidth]{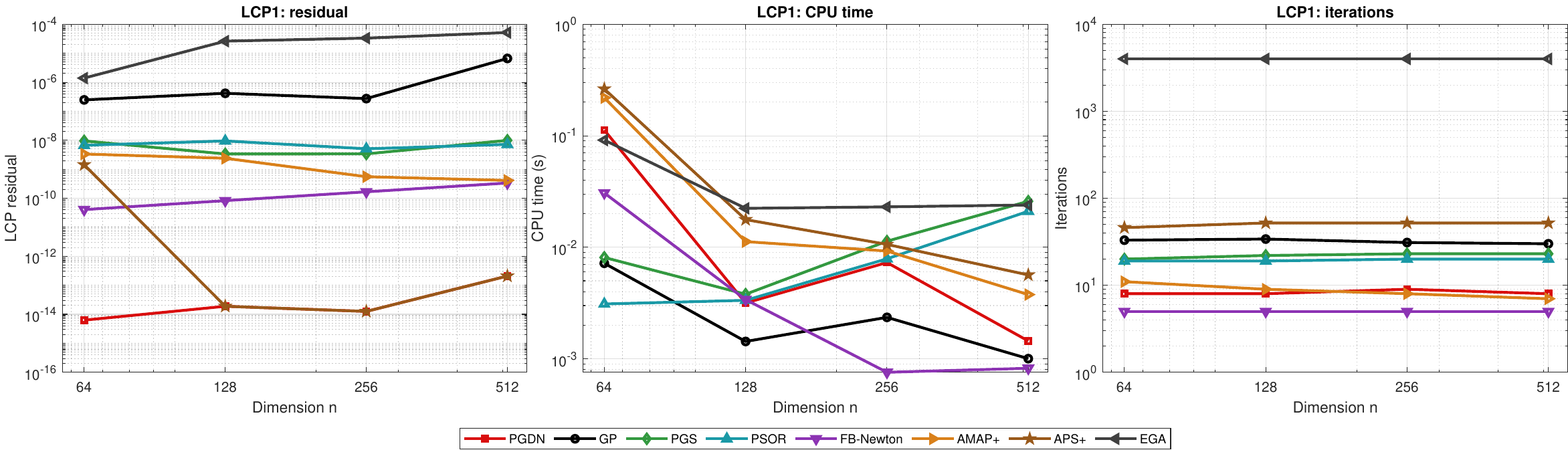}
    \caption{LCP1.}
    \label{fig:direct_lcp1}
  \end{subfigure}\par\smallskip
  \begin{subfigure}{\textwidth}
    \centering
    \includegraphics[width=.85\textwidth]{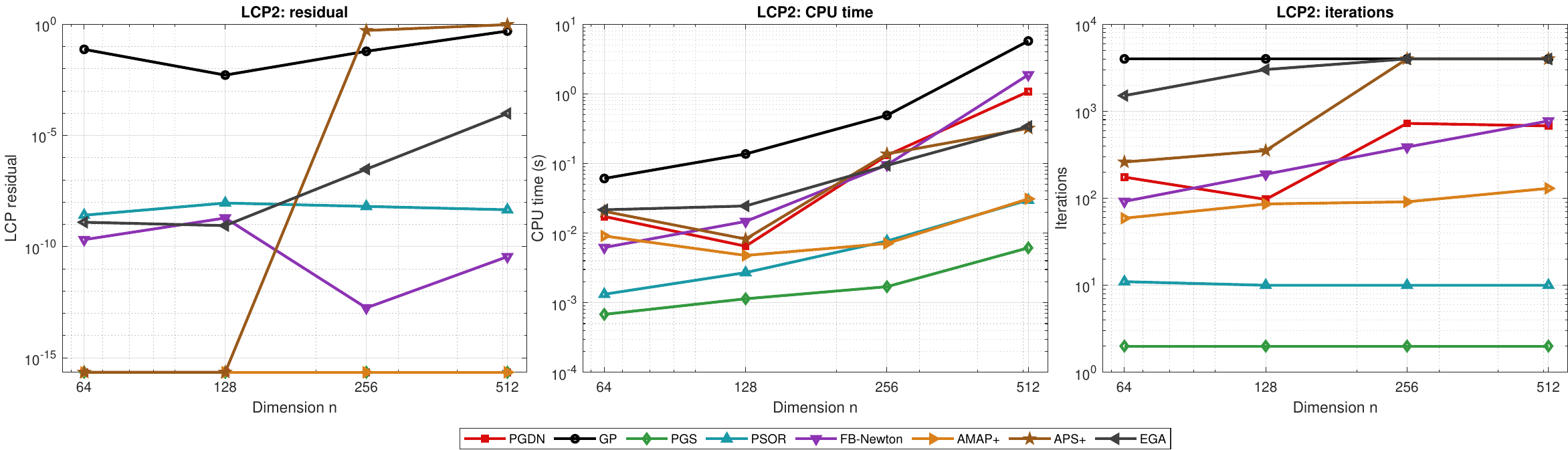}
    \caption{LCP2.}
    \label{fig:direct_lcp2}
  \end{subfigure}\par\smallskip
  \begin{subfigure}{\textwidth}
    \centering
    \includegraphics[width=.85\textwidth]{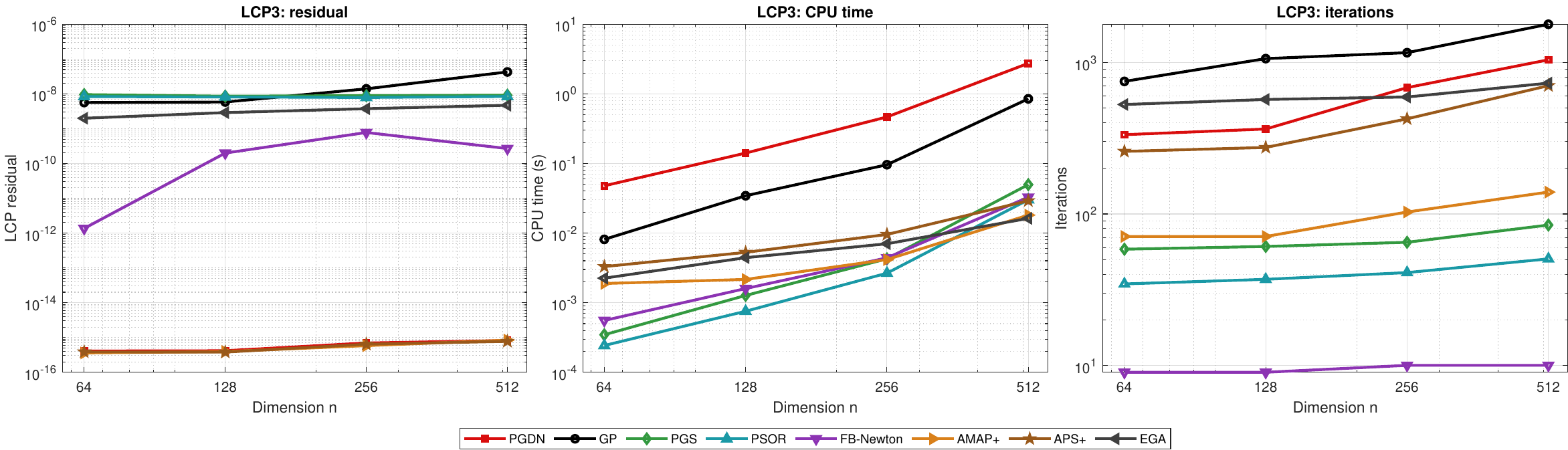}
    \caption{LCP3.}
    \label{fig:direct_lcp3}
  \end{subfigure}
  \caption{Direct solution of the three LCP benchmarks.}
\end{figure}

We compare PGDN with projected Gauss--Seidel (PGS) and projected SOR (PSOR)
\cite{cryer1971solution,cottle1992linear}, a smoothed Fischer--Burmeister
Newton method \cite{kanzow1999jacobian}, and GP
\cite{beck2011linearly}. The latter applies projected gradient with the
backtracking test to our complementarity-constrained
least-squares formulation, without active-set refinement. We also include
AMAP+, APS+, and EGA from \cite{alcantara2025global}, using the authors'
implementation in \texttt{lcp\_safp\_solver-main}. For the affine matrix
$[M,-I]$, AMAP+ uses the metric $Q=(MM^{\top}+I)^{-1}$, whereas APS+ uses
$Q=I$; both employ extrapolation and component identification.

The generated data are initially divided by $s:=\|M\|_1/\sqrt n$. For LCP2,
both normalized quantities are subsequently multiplied by $s$, thereby
recovering the original unscaled benchmark. This positive common scaling
leaves the solution set in $z$ unchanged, while the complementary slack is
scaled by the same factor. We use
\(n\in\{64,128,256,512\}\). LCP1 and LCP2 are deterministic, whereas the
reported LCP3 values are medians over ten independent instances. The stopping
residuals for every method are evaluated with $w:=Mz-b$ recomputed from the
returned $z$. The stopping parameters are set to
\(\texttt{maxIter}=4000\), \(\texttt{tolRes}=10^{-8}\), and
\(\texttt{tolStep}=10^{-10}\). We report CPU time, the natural residual
\[
  \operatorname{res}_{\rm nat}
  :=
  \frac{\|\min\{z,w\}\|_2}{\max\{1,\|b\|_2\}},
\]
and the LCP residual
\[
  \operatorname{res}_{\rm LCP}:=\max\Bigg\{ 
  \frac{\|\min\{z,0\}\|_2}{\max\{1,\|z\|_2\}},
  \frac{\|\min\{w,0\}\|_2}{\max\{1,\|w\|_2\}}, \frac{|z^{\top}w|}{\max\{1,\|z\|_2\|w\|_2\}}
  \Bigg\}. 
\]
The results are shown in Figures~\ref{fig:direct_lcp1}--\ref{fig:direct_lcp3}.

The results depend strongly on the benchmark structure. On LCP1, PGDN needs
only $8$--$9$ outer iterations and attains LCP residuals between
$6.3\times10^{-15}$ and $2.1\times10^{-13}$. FB-Newton converges in five
iterations, while PGS, PSOR, AMAP+, and APS+ also satisfy the prescribed LCP
residual tolerance. In contrast, GP stops at the affine-residual tolerance
but not at the stricter LCP criterion, and EGA reaches the iteration limit.

The triangular LCP2 strongly favors coordinate methods: PGS obtains zero
reported residual in two sweeps, and PSOR takes only $10$--$11$ sweeps. PGDN
and AMAP+ also solve every tested instance, with zero reported residual. The
performance of the two projection metrics differs markedly on this problem:
APS+ solves the cases $n=64$ and $128$, but reaches the iteration limit for
$n=256$ and $512$, whereas AMAP+ remains successful. FB-Newton solves all
four instances but its iteration count grows from $92$ to $773$. GP fails
to reach the requested accuracy within $4000$ iterations, and EGA does so for
the two larger dimensions.

On the random LCP3 instances, PGDN, AMAP+, and APS+ have a $100\%$ success
rate and median LCP residuals near machine precision at every dimension.
AMAP+ and APS+ are substantially faster than PGDN here: at $n=512$, their
median CPU times are $0.023$ and $0.043$ seconds, compared with $3.50$ seconds
for PGDN. PGS, PSOR, FB-Newton, and EGA also solve all instances, although at
larger median residuals ranging from $1.4\times10^{-12}$ to
$9.5\times10^{-9}$. GP reaches
the tolerance on $90\%$, $90\%$, $40\%$, and $10\%$ of the instances for
$n=64,128,256$, and $512$, respectively. Thus, the
active-set refinement in PGDN is effective in recovering highly accurate
solutions, but its linear algebra is not competitive with the specialized
projection and coordinate methods on these random LCPs.

\section{Concluding Remarks}\label{sec:conclusion}
We proposed a safeguarded projected-gradient framework for GAVEs and LCPs
that combines a common first-order iteration with a linear-system refinement
on a selected polyhedral face. We established conditions under which every projected
stationary point is a global solution, together with global convergence
results. Under error-bound conditions, we further proved linear convergence
to the solution set and derived explicit bounds for active-face identification.
Under \eqref{eq:selected_face_refinement_solution}, these bounds yield finite
termination at a zero-residual solution. Permanent activity-vector
identification follows when
exactly one face intersects the solution set. The numerical
experiments demonstrate high final accuracy across the tested problems,
although computational performance depends on the problem structure and the
specialized comparison methods. Future work includes inconsistent GAVEs and
extensions to nonlinear complementarity-constrained problems.

\bibliographystyle{plain}
\bibliography{mybibfile}

\end{document}